\documentclass[a4paper,12pt]{amsart}

\usepackage[ansinew]{inputenc}
\usepackage[english]{babel}
\usepackage[T1]{fontenc}
\usepackage{csquotes}
\usepackage{amsmath,amssymb,amsthm, mathtools, bbm, ifxetex}
\usepackage{trfsigns}
\usepackage{xspace}
\usepackage{geometry}
\usepackage{enumerate}
\usepackage{hyperref}
\usepackage[foot]{amsaddr}
\usepackage{todonotes}
\usepackage{xcolor}
\numberwithin{equation}{section}
\allowdisplaybreaks

\newtheorem{satz}{Theorem}[section]
\newtheorem{proposition}[satz]{Proposition}
\newtheorem{lemma}[satz]{Lemma}

\newtheorem*{bemerkung*}{Remark}

\theoremstyle{remark}

\newtheorem*{beispiel*}{Example}

\DeclarePairedDelimiter	{\abs}		{\lvert}	{\rvert}
\DeclarePairedDelimiter	{\norm}		{\lVert}	{\rVert}

\renewcommand{\epsilon}{\varepsilon}

\def\R{{\mathbb R}}

\def\cs{\circ}
\def\Rns{\mathbb{R}^{n \times n}_\mathrm{sym}}
\def\pis{\mathrm{sym}}
\def\vec{\mathrm{vec}}
\def\tr{\mathrm{tr}}
\def\pist{\pi_\mathrm{St}}
\def\pis{\pi_\mathrm{sym}}
\def\rank{\mathrm{rank}}

\def\be{\begin{equation}}
\def\en{\end{equation}}
\def\bee{\begin{eqnarray*}}
\def\ene{\end{eqnarray*}}

\begin{document}

\title{Higher order concentration on Stiefel and Grassmann manifolds}

\author{Friedrich G\"{o}tze}
\address{Friedrich G\"otze, Faculty of Mathematics, Bielefeld University, Germany}
\email{goetze@math.uni-bielefeld.de}
\author{Holger Sambale}
\address{Holger Sambale, Faculty of Mathematics, Ruhr University Bochum, Germany}
\email{holger.sambale@rub.de}

\begin{abstract}
    We prove higher order concentration bounds for functions on Stiefel and Grassmann manifolds equipped with the uniform distribution. This partially extends previous work for functions on the unit sphere. Technically, our results are based on logarithmic Sobolev techniques for the uniform measures on the manifolds. Applications include Hanson--Wright type inequalities for Stiefel manifolds and concentration bounds for certain distance functions between subspaces of $\mathbb{R}^n$.
\end{abstract}

\subjclass{Primary 60E15, Secondary 28C10, 58C35}
\keywords{Concentration of measure phenomenon, Stiefel manifold, Grassmann manifold, logarithmic Sobolev inequality, Hanson-Wright inequality}
\thanks{This research was supported by the German Research Foundation via CRC 1283.}

\date{\today}

\maketitle

\section{Introduction}

In recent years, functions on the Stiefel manifold (i.\,e., the collection of all $d$-tuples of orthonormal vectors in $\mathbb{R}^n$) and on the closely related Grassmann manifold (the collection of all $d$-dimensional linear subspaces of $\mathbb{R}^n$) have attracted increasing attention from various fields of research. A central reason is that they admit a wealth of applications in various directions like data analysis, subspace estimation, computer vision or statistical learning. In addition, they have also been studied from more theoretical points of view, in particular in convex geometry but quite recently also in large deviations theory. For a selection of various results, see e.\,g.\ \cite{TVC08,HL08,WLT11,ZZJH18,LLY20} (data analysis) and \cite{LPT06,GKZ21,KP21} (including large deviation principles).

A frequent question which naturally arises in many applications is how to control the fluctuations of some (Stiefel or Grassmann) functional around a typical value like its expectation, i.\,e., one asks for suitable concentration of measure results. Often times, the functional is of Lipschitz-type, which in classical situations (e.\,g., functions of independent sub-Gaussian random variables or in presence of a log-Sobolev inequality) leads to sub-Gaussian tail bounds. For a brief review especially adapted to Stiefel and Grassmann manifolds, cf.\ Section \ref{CR} below. However, this approach fails if the functionals under consideration are no longer Lipschitz. Still, one may hope for useful concentration bounds if they are Lipschitz of a ``higher order'', which typically means that their higher order derivatives are absolutely bounded. A classical example for order two are quadratic forms. In this situation, various questions of interest include exponential moment bounds, refined tail bounds with various levels of decay and centering around non-deterministic quantities which may correspond to a suitable decomposition.

In this note, we provide higher order concentration results for Stiefel and Grassmann manifolds equipped with the uniform distribution. Higher order concentration has been studied in various settings in the past two decades, e.\,g.\ in \cite{La06} (polynomials in independent Gaussian variables) \cite{AW15} (measures satisfying certain Sobolev-type inequalities) as well as \cite{KV00,SS12,AL12,KL15,GSS21a,GSS21b}. In this paper, we especially continue the line of research begun in \cite{BCG17,GSS19}, where second and higher order concentration results for the unit sphere were established (noting that the unit sphere can be understood as a Stiefel manifold).

\subsection{Grassmann and Stiefel manifolds}
For any natural numbers $d \le n$, the Stiefel manifold $W_{n,d}$ is the set of all $d$-tupels of orthonormal vectors in $\mathbb{R}^n$. Clearly, $W_{n,d}$ may be written as
\[
W_{n,d} = \{A \in \mathbb{R}^{n \times d} \colon A^TA = I_d \},
\]
where $I_d \in \mathbb{R}^{d \times d}$ denotes the identity matrix. This is the representation we shall use throughout this note. $W_{n,d}$ is manifold of dimension $\mathrm{dim}(W_{n,d}) = nd - d(d+1)/2$. Obviously, $W_{n,1} = S^{n-1}$ (the unit sphere), $W_{n,n} = O(n)$ (the orthogonal group), and $W_{n,n-1}$ can be identified with the special orthogonal group $SO(n) = \{O \in O(n) \colon \mathrm{det}(O) = 1\}$.

We may equip $W_{n,d}$ with the subspace topology and distances inherited from $\mathbb{R}^{n \times d}$, including the scalar product on $\mathbb{R}^{n\times d}$ and the induced (Hilbert--Schmidt) norm
\[
\langle A, B\rangle:=\mathrm{tr}(A^T B),\qquad \lvert A \rvert^2 \equiv \lVert A\rVert_\mathrm{HS}^ 2= \mathrm{tr}(A^TA) = \sum_{i,j} A_{ij}^2.
\]
In particular, $W_{n,d}$ is a compact topological space. The product of the orthogonal groups $O(n) \times O(d) $ acts transitively on $W_{n,d}$ by the two-sided multiplication $O_n \times O_d \mapsto O_nAO_d$, turning $W_{n,d}$ into a homogeneous space. Hence, it may be equipped with a unique invariant (Haar) probability measure $\mu_{n,d}$ in the sense that if $A \sim \mu_{n,d}$ (i.\,e., $A$ has distribution $\mu_{n,d}$), $O_nAO_d \sim \mu_{n,d}$ for any $O_n \in O(n)$, $O_d \in O(d)$. We call $\mu_{n,d}$ the uniform distribution on $W_{n,d}$. If $G = (G_{ij})_{i,j=1}^{d,n}$ is an $n \times d$ random matrix whose entries are i.i.d.\ standard normal, then $G(G^TG)^{-1/2} \sim \mu_{n,d}$, see \cite[Lemma 3.1]{KPT20}.

There are several ways of introducing the closely related Grassmann manifold, i.\,e.\ the set of all $d$-dimensional subspaces of $\mathbb{R}^n$. For an overview, cf.\ e.\,g.\ \cite{BZA20}. For instance, identifying a subspace with its basis (which is unique up to orthogonal transformations), we may understand the Grassmann manifold as the quotient $W_{n,d}/O(d)$, identifying any two $X, X' \in W_{n,d}$ such that $X' = XO$ for some $O \in O(d)$. However, the calculus of (higher order) derivatives on these equivalence classes turns out to be troublesome, and therefore, we choose a different approach.

For any $A \in W_{n,d}$, $\pist(A) := AA^T \equiv P_A$ is a projection matrix of rank $d$ (more precisely, the projection onto the subspace a basis of which is given by the columns of $A$), and for any $O_d \in O(d)$, we have $P_{AO_d} = P_A$. Therefore, identifying elements of the Grassmannian with projection matrices, we may define
\[
G_{n,d} := \{P \in \Rns \colon P^2 = P, \ \rank(P) = d \},
\]
where $\Rns$ denotes the space of the symmetric $n \times n$ matrices. $G_{n,d}$ is a manifold of dimension $\mathrm{dim}(G_{n,d}) = d(n-d)$. If $d=1$, we get back the half-sphere (where $\theta$ and $-\theta$ are identified), which can also be regarded as the projective space $\mathbb{R}P^{n-1}$. Depending on the situation, we will either regard $G_{n,d}$ as a submanifold of $\Rns$ or $\mathbb{R}^{n \times n}$ (the latter is often more convenient when taking derivatives), equipping it with the inherited topology and distances similarly to the case of the Stiefel manifold (in particular, $G_{n,d}$ is compact).

Many of the properties of Stiefel manifolds can be extended to Grassmann manifolds. Especially, the group action $O(n) \ni O_n \mapsto P_{O_nA}$ turns $G_{n,d}$ into a symmetric (not just homogeneous) space. We denote the uniform distribution on $G_{n,d}$ by $\nu_{n,d}$. Clearly, $\nu_{n,d}$ is the pushforward of $\mu_{n,d}$ under the map $\pist$. If $G = (G_{ij})_{i,j=1}^{d,n}$ is an $n \times d$ random matrix whose entries are i.i.d.\ standard normal, it follows from the discussion above that $G(G^TG)^{-1}G^T \sim \nu_{n,d}$.

\subsection{Concentration results}\label{CR}
If $f \colon W_{n,d} \to \mathbb{R}$ is $L$-Lipschitz with mean $\mu_{n,d}(f)$ with respect to $\mu_{n,d}$, a standard concentration result (cf.\ e.\,g.\ \cite[p.\ 27]{Le01}) yields
\begin{equation}\label{Lipschitz}
\mu_{n,d}(f-\mu_{n,d}(f) \ge t) \le \exp(-(n-1)t^2/(8L^2)).
\end{equation}
By switching to the pushforward, the same result also holds for $(G_{n,d},\nu_{n,d})$ with the constant $8$ replaced by $16$ (cf.\ Section \ref{sec:Sob} for details, in particular Lemma \ref{pistLip}).

The aim of this note is to extend results of type \eqref{Lipschitz} to higher orders, involving $\mathcal{C}^k$ functions and derivatives up to order $k$. Here, a $\mathcal{C}^k$ function on $W_{n,d}$ (likewise, $G_{n,d}$) may be understood as a function which admits an extension to some open neighborhood of the manifold which is $\mathcal{C}^k$-smooth. Let us first fix some notation. If $A = (a_{i_1 \ldots i_k}) \in \mathbb{R}^{m^k}$ is any matrix or tensor of order $k$, we write
\begin{align*}
\lVert A \rVert_\mathrm{HS} &:= \Big(\sum_{i_1, \ldots, i_k} a_{i_1 \ldots i_k}^2\Big)^{1/2},\\
\lVert A \rVert_\mathrm{op} &:= \sup\Big\{\sum_{i_1, \ldots, i_k} a_{i_1 \ldots i_k} x^{(1)}_{i_1} \cdots x^{(k)}_{i_k} \colon x^{(j)} \in S^{m-1} \ \forall j \Big\}
\end{align*}
for the respective Hilbert--Schmidt and operator type norms. For any function $g \colon W_{n,d} \to \mathbb{R}$ and any $p \ge 1$, we denote by $\lVert g \rVert_p$ the $L^p$ norm of $g$ with respect to $\mu_{n,d}$. If $g$ is matrix-valued (e.\,g., some tensor of derivatives), we use the short-hand notation
\[
\lVert g \rVert_{\mathrm{HS},p} := \lVert \lVert g \rVert_\mathrm{HS} \rVert_p = \Big(\int_{W_{n,d}} \lVert g \rVert_\mathrm{HS}^p d\mu_{n,d}\Big)^{1/p}
\]
and similarly $\lVert g \rVert_{\mathrm{op},p}$.

In Section \ref{sec:DerSt}, we will introduce a notion of differentiability for functions on $W_{n,d}$. In particular, using the intrinsic (Stiefel) gradient $\nabla_W f \in \mathbb{R}^{n \times d}$ and assuming $f \colon W_{n,d} \to \mathbb{R}$ to be $\mathcal{C}^1$-smooth, we may reformulate \eqref{Lipschitz} as
\begin{equation}\label{LipschitzAbl}
\mu_{n,d}(f-\mu_{n,d}(f) \ge t) \le \exp\Big(-\frac{(n-1)t^2}{8\lVert \nabla_W f \rVert_{\mathrm{HS},\infty}^2}\Big).
\end{equation}
The same result holds for $\mathcal{C}^1$-smooth functions on $(G_{n,d},\nu_{n,d})$, involving the Grassmann gradient $\nabla_G f \in \Rns$ and with $8$ replaced by $16$.

Our first theorem complements \eqref{LipschitzAbl} by a second order concentration bound for functions on Stiefel manifolds of any order $d \le n-1$ in the spirit of \cite{BCG17} (cf.\ also \cite{SS21}, where this approach to second order concentration has been generalized and adapted to a wealth of different situations). In addition to the intrinsic gradient, this also involves an intrinsic Hessian $f''_W \in \mathbb{R}^{nd \times nd}$, a notion to be made precise in Section \ref{sec:DerSt} as well.

\begin{satz}\label{Conc2ndOr}
Let $f \colon W_{n,d} \to \mathbb{R}$ be a $\mathcal{C}^2$-smooth function such that $\mu_{n,d}(f) = 0$.
\begin{enumerate}
    \item Assuming $\lVert \nabla_W f \rVert_{\mathrm{HS},2} \le 2/\sqrt{n-2}$ and $\lVert f''_W \rVert_{\mathrm{op},\infty} \le 1$, we have
    \[
    \int_{W_{n,d}} \exp\Big(\frac{n-2}{32e} |f|\Big) d\mu_{n,d} \le 2.
    \]
    \item For $C = 16e^2/\log(2)$ and any $t \ge 0$,
    \[
    \mu_{n,d}(|f-\mu_{n,d}(f)| \ge t) \le 2 \exp\Big(-\frac{n-2}{C}\min\Big(\frac{t^2}{\lVert \nabla_W f \rVert_{\mathrm{HS},2}^2}, \frac{t}{\lVert f_W'' \rVert_{\mathrm{op},\infty}}\Big)\Big).
    \]
    \item If $\mu_{n,d}(\nabla_W f) = 0$, the bounds in (1) and (2) continue to hold with $\lVert \nabla_W f \rVert_{\mathrm{HS},2}$ replaced by $\sqrt{8/(n-2-8d)}\lVert f_W'' \rVert_{\mathrm{HS},2}$.
\end{enumerate}
\end{satz}

Note that in (3), the integral $\mu_{n,d}(\nabla_W f)$ has to be understood componentwise. In fact, the condition in (3) can also be modified to $\lVert f_W'' \rVert_{\mathrm{HS},2} \le b$ similarly as in \cite[Theorem 1.1]{BCG17} (in particular, arriving at $b$-dependent bounds). We skip the details. With only minor modifications, these results also hold for functions on Grassmann manifolds:

\begin{satz}\label{Conc2ndOrGr}
Let $f \colon G_{n,d} \to \mathbb{R}$ be a $\mathcal{C}^2$-smooth function such that $\nu_{n,d}(f) = 0$.
\begin{enumerate}
    \item Assuming $\lVert \nabla_G f \rVert_{\mathrm{HS},2} \le \sqrt{8/(n-2)}$ and $\lVert f''_G \rVert_{\mathrm{op},\infty} \le 1$, we have
    \[
    \int_{G_{n,d}} \exp\Big(\frac{n-2}{64e} |f|\Big) d\mu_{n,d} \le 2.
    \]
    \item For $C = 32e^2/\log(2)$ and any $t \ge 0$,
    \[
    \nu_{n,d}(|f-\mu_{n,d}(f)| \ge t) \le 2 \exp\Big(-\frac{n-2}{C}\min\Big(\frac{t^2}{\lVert \nabla_G f \rVert_{\mathrm{HS},2}^2}, \frac{t}{\lVert f_G'' \rVert_{\mathrm{op},\infty}}\Big)\Big).
    \]
    \item If $\nu_{n,d}(\nabla_G f) = 0$, the bounds in (1) and (2) continue to hold with $\lVert \nabla_G f \rVert_{\mathrm{HS},2}$ replaced by $\sqrt{16/(n-2-16d)}\lVert f_G'' \rVert_{\mathrm{HS},2}$.
\end{enumerate}
\end{satz}

In general, once we have a function $f$ which is defined and smooth in a neighborhood of $W_{n,d}$ (which may be achieved by choosing a suitable extension of $f$), we may also formulate concentration of measure results involving the usual (Euclidean) derivatives only. The same holds for functions defined in a neighborhood of $G_{n,d}$ in $\Rns$ or $\mathbb{R}^{n \times n}$.

As shown in Sections \ref{sec:DerSt} and \ref{sec:DerGr}, we always have $|\nabla_W f|, |\nabla_G f| \le |\nabla f|$. In particular, results depending on Euclidean derivatives are somewhat less accurate than bounds involving intrinsic derivatives. On the other hand, the underlying calculus is typically much less involved, especially as the order of the derivatives increases, and in many situations, the usual derivatives may already be sufficient for meaningful concentration bounds. Here, we have the following result, which complements \cite{BGS19}.

\begin{satz}\label{ConckthOr}
Let $f$ be a real-valued $\mathcal{C}^k$-smooth function defined in some neighborhood of $W_{n,d}$ such that $\mu_{n,d}(f) = 0$.
\begin{enumerate}
    \item Assuming $\lVert f^{(\ell)} \rVert_{\mathrm{op},2} \le (4/(n-2))^{(k-\ell)/2}$ for any $\ell = 1, \ldots, k-1$ and $\lVert f^{(k)} \rVert_{\mathrm{op},\infty}\linebreak[2] \le 1$, we have
    \[
    \int_{W_{n,d}} \exp\Big(\frac{n-2}{32e} |f|^{2/k}\Big) d\mu_{n,d} \le 2.
    \]
    \item For $C = 4e^2/\log(2)$ and any $t \ge 0$,
    \[
    \mu_{n,d}(|f| \ge t) \le 2 \exp\Big(-\frac{n-2}{Ck^2}\min\Big(\min_{\ell = 1, \ldots, k-1}\frac{t^{2/\ell}}{\lVert f^{(\ell)} \rVert_{\mathrm{op},2}^{2/\ell}}, \frac{t^{2/k}}{\lVert f^{(k)} \rVert_{\mathrm{op},\infty}^{2/k}}\Big)\Big).
    \]
\end{enumerate}
\end{satz}

The Grassmann version of this theorem reads as follows:

\begin{satz}\label{ConckthOrGr}
Let $f$ be a real-valued $\mathcal{C}^k$-smooth function defined in some neighborhood of $G_{n,d}$ such that $\nu_{n,d}(f) = 0$.
\begin{enumerate}
    \item Assuming $\lVert f^{(\ell)} \rVert_{\mathrm{op},2} \le (8/(n-2))^{(k-\ell)/2}$ for any $\ell = 1, \ldots, k-1$ and $\lVert f^{(k)} \rVert_{\mathrm{op},\infty}\linebreak[2] \le 1$, we have
    \[
    \int_{G_{n,d}} \exp\Big(\frac{n-2}{64e} |f|^{2/k}\Big) d\nu_{n,d} \le 2.
    \]
    \item For $C = 8e^2/\log(2)$ and any $t \ge 0$,
    \[
    \nu_{n,d}(|f| \ge t) \le 2 \exp\Big(-\frac{n-2}{Ck^2}\min\Big(\min_{\ell = 1, \ldots, k-1}\frac{t^{2/\ell}}{\lVert f^{(\ell)} \rVert_{\mathrm{op},2}^{2/\ell}}, \frac{t^{2/k}}{\lVert f^{(k)} \rVert_{\mathrm{op},\infty}^{2/k}}\Big)\Big).
    \]
\end{enumerate}
\end{satz}

Examples where these results are applied to concrete situations and are briefly compared to related bounds known from the literature are given in Section \ref{Sec:PolCh}.

\subsection{Overview}
In view of the parametrizations of the manifolds via matrices, we first provide a brief review on some basic facts about derivatives in matrix arguments and related topics in Section \ref{sec:Nabf}. First and second order intrinsic derivatives for functions on Stiefel manifolds are introduced and discussed in Sections \ref{sec:DerSt} and \ref{sec:2OrMoG}, including a number of inequalities relating first and second order. These results are adapted to Grassmann manifolds in Section \ref{sec:DerGr}. The proofs of our main results are given in Section \ref{sec:Sob}. Applications to polynomial chaos (in particular of order $2$), including concentration of the distance to a given subspace, are provided in Section~\ref{Sec:PolCh}.

\section{Basic facts about matrix calculus}\label{sec:Nabf}

The aim of this section is to fix some notation and to collect some basic facts about matrix calculus mostly for the sake of reference. Most of the results are elementary and will be stated without proofs.

If $A = (A_{ij})$ is an $n \times m$ matrix (i.\,e., $A \in \mathbb{R}^{n \times m}$), we may vectorize it by setting
\[
\mathrm{vec}(A)  := (A_{11}, \ldots, A_{n1}, A_{12}, \ldots, A_{n2}, \ldots, A_{1m},\ldots, A_{nm})^T,
\]
i.\,e.\ $\mathrm{vec}(A)$ is the vector in $\mathbb{R}^{nm}$ with $m$ $n$-blocks corresponding to the columns of $A$. Occasionally, we will also need the inverse operation of $\mathrm{vec}$, which we denote by $\mathrm{mat}$. For $B \in \mathbb{R}^{nd \times nd}$ and $U, V \in \mathbb{R}^{n \times d}$, we introduce the short-hand notation
\begin{equation}\label{shn}
    BV := B\mathrm{vec}(V),\qquad \langle BV, U\rangle := \langle B\mathrm{vec}(V), \mathrm{vec}(U) \rangle.
\end{equation}
There is an $nm \times nm$ permutation matrix $K_{n,m}$ such that for any $A \in \mathbb{R}^{n \times m}$,
\begin{equation}\label{CommMat}
\mathrm{vec}(A^T) = K_{n,m}\mathrm{vec}(A).
\end{equation}
$K_{n,m}$ is uniquely defined by \eqref{CommMat} and is called the $(n,m)$ commutation matrix. Note that $K_{n,m}^T = K_{m,n}$.

Vectorization is closely related to the Kronecker product $A \otimes B$ for any $A \in \mathbb{R}^{n \times m}$ and any $B \in \mathbb{R}^{p \times q}$. The following lemma lists some of its elementary properties.

\begin{lemma}\label{PropKronProd}
\begin{enumerate}
    \item The Kronecker product is bilinear and associative.
    \item We have $(A \otimes B)^T = A^T \otimes B^T$.
    \item $A \otimes B$ is invertible whenever $A$ and $B$ are invertible, and in this case, $(A \otimes B)^{-1} = A^{-1} \otimes B^{-1}$.
    \item For any $A \in \mathbb{R}^{n \times m}, B \in \mathbb{R}^{p \times q}, C \in \mathbb{R}^{m \times r}, D \in \mathbb{R}^{q \times s}$, we have $(A \otimes B)(C \otimes D) = (AC) \otimes (BD)$.
    \item For any $A \in \mathbb{R}^{n \times m}, B \in \mathbb{R}^{m \times p}$, we have $\mathrm{vec}(AB) = (I_p \otimes A) \mathrm{vec}(B) = (B^T \otimes I_n) \mathrm{vec}(A)$.
    \item For any $A \in \mathbb{R}^{n \times m}, B \in \mathbb{R}^{m \times p}, C \in \mathbb{R}^{p \times q}$, we have $\mathrm{vec}(ABC) = (C^T \otimes A)\mathrm{vec}(B) = (I_q \otimes AB)\mathrm{vec}(C) = (C^TB^T \otimes I_n) \mathrm{vec}(A)$.
    \item For any $A \in \mathbb{R}^{n \times m}$ and any $B \in \mathbb{R}^{p \times q}$, we have $K_{p,n} (A \otimes B) K_{m,q} = B \otimes A$.
\end{enumerate}
\end{lemma}

If $f \colon \mathbb{R}^{n\times m} \to \mathbb{R}^{p\times q}$ is any differentiable (possibly matrix-valued) function of $X \in \mathbb{R}^{n \times m}$, we define
\[
Df(X) := D\mathrm{vec}(f(X)) := \frac{d f(X)}{dX} := \frac{d \mathrm{vec}(f(X))}{d \mathrm{vec}(X)},
\]
which is a $pq \times mn$ matrix. Let us recall the usual differentiation rules for derivatives in matrix arguments together with the derivatives of several standard types of functions.

\begin{lemma}\label{MatrixDeriv}
Let $X \in \mathbb{R}^{n\times m}$, $A \in \mathbb{R}^{p \times n}$ and $B \in \mathbb{R}^{m \times q}$.
\begin{enumerate}
    \item Let $f(X) \in \mathbb{R}^{p\times q}$, $g(X) \in \mathbb{R}^{q\times k}$ be differentiable functions of $X$. Then, the \emph{product rule} holds:
    \[
    D(f(X)g(X)) = (g(X)^T \otimes I_p)Df(X) + (I_k \otimes f(X)) Dg(X).
    \]
    \item For differentiable functions $f$ on $\mathbb{R}^{q \times k}$ and $g(X) \in \mathbb{R}^{q\times k}$, the usual \emph{chain rule} holds:
    \[
    D[f(g(X))] = Df(g(X)) \cdot Dg(X).
    \]
    \item We have $\frac{dX^T}{dX} = K_{n,m}$.
    \item We have $\frac{d(AXB)}{dX} = B^T \otimes A$.
    \item We have $\frac{d\mathrm{tr}(AX)}{dX} = \mathrm{vec}(A^T)^T$.
\end{enumerate}
\end{lemma}

\section{Derivatives on Stiefel manifolds}\label{sec:DerSt}

To formally introduce intrinsic derivatives of first and second order on Stiefel manifolds, we follow and extend the spherical case as discussed in \cite{BCG17}. If $f$ is any real-valued locally Lipschitz function on some metric space $(M,d)$ (with no isolated points), we may always define the generalized modulus of the gradient of $f$ in $x \in M$
\begin{equation}\label{genmod}
|\nabla^* f(x)| = \limsup_{x' \rightarrow x} \frac{|f(x)-f(x')|}{d(x,x')}.
\end{equation}
In particular, we may use \eqref{genmod} for functions $f(X)$ on the Stiefel manifold $W_{n,d}$ together with the Euclidean (Hilbert--Schmidt) metric $d(X,X') = |X-X'|$. Here we will always write $|\nabla^*f(X)| \equiv |\nabla_Wf(X)|$, cf.\ the discussion after \eqref{defderSt} below. Note we could also use the (point-dependent) canonical instead of the Euclidean metric, which leads to different notions of differentiability, cf.\ e.\,g.\ \cite{EAS98}.

To introduce a notion of differentiability of a function $f$ on $G_{n,d}$ which is consistent with \eqref{genmod}, recall that the tangent space in $A \in W_{n,d}$ is given by
\[
T_A:= \{  N \in \mathbb{R}^{n \times d} : A^T N + N^T A = 0\} = \{  N \in \mathbb{R}^{n \times d} : A \cs N = 0\},
\]
where for any $M, N \in \mathbb{R}^{n \times d}$ (or $M, N \in \mathbb{R}^{d \times d}$ as we shall also need later on), $M \cs N$ denotes the ``symmetric product''
\[
M \cs N := \frac{1}{2} (M^T N + N^T M),
\]
which is a symmetric $d \times d$ matrix (for $d=1$, this reduces to the Euclidean scalar product). Hence, $T_A$ is the set of all the matrices $N \in \mathbb{R}^{n \times d}$ such that $A^TN$ is antisymmetric. Now, a function $f \colon W_{n,d} \to \mathbb{R}$ is differentiable at $A \in W_{n,d}$ if it admits a Taylor expansion
\begin{equation}\label{defderSt}
f(A') = f(A) + \langle M,A' - A\rangle + o\big(|A' - A|\big) \quad {\rm as} \ \ A' \rightarrow A, \ \ A' \in W_{n,d}
\end{equation}
with some $M \in \R^{n \times d}$. The unique $M_0$ of smallest (Euclidean) length among all such $M$ is called the intrinsic derivative or gradient of $f$ at $A$ and is denoted $\nabla_W f(A)$. The length of $\nabla_W f(A)$ agrees with \eqref{genmod}.

Once $M \in \mathbb{R}^{n \times d}$ satisfies \eqref{defderSt}, any matrix $M-B$ satisfies \eqref{defderSt} as well if
\[
\langle B, A'-A \rangle = o(|A'-A|)
\]
as $A' \to A$, $A, A' \in W_{n,d}$. The latter is equivalent to $\langle B, N \rangle = 0$ for all $N \in T_A$, i.\,e.\ $B \in T_A^\perp$. Therefore, the minimization problems translates into
\[
\lVert M - B \rVert \to \mathrm{min}\qquad \text{over all $B \in T_A^\perp$,}
\]
which is solved uniquely for the orthogonal projection $M$ onto $T_A$.

Often, it is convenient to consider functions $f$ which are defined and smooth in an open neighborhood of $W_{n,d}$, which may be achieved by choosing a suitable extension of the function under consideration. In this case, we may take the Euclidean derivative $Df(A) \in \mathbb{R}^{1 \times nd}$ of $f$ in $A \in W_{n,d}$ (cf.\ the previous section) and set $\nabla f(A) := \mathrm{mat}(Df(A)^T) \in \mathbb{R}^{n \times d}$ (the usual gradient rewritten as a matrix). Clearly, $\nabla f(A)$ satisfies \eqref{defderSt}, and we may project it onto $T_A$. This way, we get back $\nabla_W f(A)$.

\begin{proposition}\label{projSt}
Let $f$ be defined and $\mathcal{C}^1$-smooth in a neighborhood of $W_{n,d}$. Then, the intrinsic first derivative of $f$ at $A \in W_{n,d}$ is given by the projection onto $T_A$
\begin{equation}\label{FormelStiefelGradient}
\nabla_W f(A) = \nabla f(A) - A(A\cs \nabla f(A)).
\end{equation}
In particular, $|\nabla_W f(A)| \leq |\nabla f(A)|$ for any $A \in W_{n,d}$.
\end{proposition}

Defining the projection $\pi_A \colon \mathbb{R}^{n \times d} \to T_A$ by
\begin{equation}\label{projfSt}
\pi_AM := M - A(A \cs M),
\end{equation}
Proposition \ref{projSt} states that $\nabla_W f(A) = \pi_A \nabla f(A)$. In particular, by the contractivity of orthogonal projections, this shows that indeed, $|\nabla_W f(A)| \leq |\nabla f(A)|$.

Next, we introduce second order (intrinsic) derivatives (Hessians). For any $C^2$-smooth function $f \colon \mathbb{R}^{n \times d} \to \mathbb{R}$ at a given point $A \in W_{n,d}$, consider the Taylor expansion up to the quadratic term (using the notation \eqref{shn})
\begin{equation}\label{Taylor2ndOrderSt}
f(A') = f(A) + \langle\nabla_W f(A), A'-A\rangle + \frac{1}{2} \langle B(A'-A), A'-A \rangle + o\big(|A' - A|^2\big)
\end{equation}
as $A' \rightarrow A$, $A' \in W_{n,d}$, where $B \in \mathbb{R}^{nd \times nd}$ is some matrix. The collection of all $B$'s satisfying \eqref{Taylor2ndOrderSt} represents an affine subspace of $\mathbb{R}^{nd \times nd}$. Therefore, among all of them, there exists a unique matrix of smallest Hilbert--Schmidt norm. It is called the (intrinsic) second derivative of $f$ at the point $X$ and will be denoted $f_W''(A)$.

If $f$ is $C^2$-smooth function on some open neighborhood of $W_{n,d}$, by the usual (Euclidean) Taylor expansion, \eqref{Taylor2ndOrderSt} holds with
\begin{equation}\label{MatrixBSt}
B = f''(A) - (A \cs \nabla f(A)) \otimes I_n,
\end{equation}
where $f''(A)$ is the $nd \times nd$ matrix of the usual (Euclidean) second order derivatives of $f$ at $A$. Indeed, plugging \eqref{FormelStiefelGradient} and \eqref{MatrixBSt} into \eqref{Taylor2ndOrderSt}, this follows from
\begin{align*}
&\qquad \langle A(A \cs \nabla f(A)), A' - A\rangle + \frac{1}{2} \langle ((A \cs \nabla f(A)) \otimes I_n)(A' - A), A'-A\rangle\\
&= \frac{1}{2} \langle (A'+A)(A \cs \nabla f(A)), A' - A\rangle
= \frac{1}{2} \mathrm{tr} ((A \cs \nabla f(A)) (A'^T + A^T)(A'-A))\\
    &= \frac{1}{2} \mathrm{tr} ((A \cs \nabla f(A)) (A'^TA-A^TA')) = 0,
\end{align*}
where in the first identity we have used that by \eqref{shn} and Lemma \ref{PropKronProd} (5),
\[
((A \cs \nabla f(A)) \otimes I_n)(A'-A) = \mathrm{vec}((A'-A)(A \cs \nabla f(A))),
\]
the third step follows from $A, A' \in W_{n,d}$ and in the last one we use that the trace of a product of a symmetric and an antisymmetric matrix is $0$.

Given $C \in \mathbb{R}^{nd \times nd}$, the matrix $B - C$ satisfies \eqref{Taylor2ndOrderSt} if and only if
\[
\langle C(A' - A), A'-A\rangle = o\big(|A'-A|^2\big)
\]
for $A' \to A$, $A' \in W_{n,d}$. Similarly to the first order case, this is equivalent to $\langle CX, X \rangle = 0$ for all $X \in T_A$. This condition defines a linear subspace $L$ of $\mathbb{R}^{nd \times nd}$, and the minimization problem translates into
\[
\lVert B - C \rVert \to \mathrm{min}\qquad \text{over all $C \in L$,}
\]
which is solved uniquely for the orthogonal projection of $B$ onto the linear space $L^\perp$ of all matrices orthogonal to $L$. Since $B$ is symmetric, we may restrict ourselves to symmetric matrices. As a result, we obtain the following description.

\begin{proposition}\label{2ndOrDerSt}
The intrinsic second derivative of $f$ at each $A \in W_{n,d}$ is the symmetric matrix which is given by the orthogonal projection
\[
f_W''(A) = P_{L^\perp} B, \qquad B = f''(A) - (A \cs \nabla f(A)) \otimes I_n,
\]
to the orthogonal complement of the linear subspace $L = L_A$ of all symmetric matrices $C$ in $\mathbb{R}^{nd \times nd}$ such that $\langle CX, X \rangle = 0$ for all $X \in T_A$. Equivalently,
\[
f_W''(A) = P_A B P_A,
\]
where $P_A= AA^T$ and $P_A B P_A$ is the $nd \times nd$ matrix which is defined by the relation
\begin{equation}\label{NotExpl}
(P_A B P_A)\vec(V) = \vec(P_A \mathrm{mat}[B\mathrm{vec}(P_AV)])
\end{equation}
for any $V \in \mathbb{R}^{n\times d}$.
\end{proposition}
 
In particular, for all $A \in W_{n,d}$ and $V \in \mathbb{R}^{n\times d}$, $f_W''(A)V \in T_A$, $f''_W(A)A = 0$, and hence $\langle f_W''(A)V, A \rangle = 0$ for all $V$. Furthermore, we have the contraction property
\[
\|f_W''(A)\|_\mathrm{HS} \, \leq \, \|f''(A) - (A \cs \nabla f(A)) \otimes I_n\|_{{\rm HS}},
\]
which also holds for the operator norm.

In the notation of \eqref{shn}, it is not hard to show that for any $V \in \mathbb{R}^{n \times d}$
\[
f''(X)V = \nabla \langle \nabla f(X), V \rangle
\]
for the usual Euclidean derivatives. The analogue for intrinsic derivatives reads as follows:

\begin{proposition}\label{Prop9.1.3St}
Given a $C^2$-smooth function $f$ on $W_{n,d}$, for all $A \in W_{n,d}$ and $V \in \mathbb{R}^{n \times d}$, we have
\[
f''_W(A)V = \nabla_W\langle \nabla_W f(A), V \rangle + P_A(\nabla_W f(A)(A \cs V)).
\]
Here, the left-hand side has to be read as $f''(W)\mathrm{vec}(V)$ and the right-hand side has to be understood as vectorized.
\end{proposition}

In particular, if $V \in T_A$, then $P_AV = V$ and $A \cs V = 0$, so that we obtain
\[
f''_W(A)V = \nabla_W\langle \nabla_W f(A), V \rangle.
\]

\begin{proof}
Consider the function $\psi_V(A) := \langle \nabla_W f(A), V \rangle$. In view of Proposition \ref{projSt}, any smooth extension of $f$ to a neighborhood of $W_{n,d}$ also yields a smooth extension of $\psi_V$ which is given by
\begin{align*}
\psi_V(X) &= \langle \nabla f(X), V \rangle - \langle X(X \cs \nabla f(X)), V \rangle\\
&= \langle \nabla f(X), V \rangle - \mathrm{tr}(V^TX(X \cs \nabla f(X))).
\end{align*}
Let us calculate $\nabla \psi_V(X)$. First, we have
\[
    \vec(\nabla \langle \nabla f(X), V \rangle) = f''(X) \mathrm{vec}(V).
\]
Next, we consider
\begin{align*}
\frac{d}{dX} [V^TX(X \cs \nabla f(X))] &= ((X \cs \nabla f(X)) \otimes I_d) \frac{d(V^TX)}{dX}\\ &\quad+ (I_d \otimes (V^TX)) \frac{d(X \cs \nabla f(X))}{dX},
\end{align*}
where we have applied Lemma \ref{MatrixDeriv} (1). By Lemma \ref{MatrixDeriv} (4),
\[
\frac{d(V^TX)}{dX} = I_d \otimes V^T,
\]
so that
\[
((X \cs \nabla f(X)) \otimes I_d) \frac{d(V^TX)}{dX} = (X \cs \nabla f(X)) \otimes V^T.
\]
Moreover, using Lemma \ref{MatrixDeriv} (1\&3), we have
\[
\frac{d(X^T\nabla f(X))}{dX} = (\nabla f(X)^T \otimes I_d)K_{n,d} + (I_d \otimes X^T)f''(X)
\]
as well as
\[
\frac{d(\nabla f(X)^TX)}{dX} = (X^T \otimes I_d)K_{n,d}f''(X) + I_d \otimes \nabla f(X)^T,
\]
and hence, using Lemma \ref{PropKronProd} (7), it follows that
\[
\frac{d(X \cs \nabla f(X))}{dX} = \frac{1}{2} (I_{d^2} + K_{d,d})((I_d \otimes \nabla f(X)^T) + (I_d \otimes X^T)f''(X)).
\]

Putting everything together, it follows by Lemma \ref{MatrixDeriv} (2\&5) and (switching from $d/(dX)$ to $\vec(\nabla)$, hence transposing) Lemma \ref{PropKronProd} (2) that
\begin{align*}
    \vec (\nabla \psi_V(X)) &= f''(X)\mathrm{vec}(V) - ((X \cs \nabla f(X)) \otimes V)\mathrm{vec}(I_d)\\
    &\hspace{-1cm}\quad - \frac{1}{2} ((I_d \otimes \nabla f(X)) + f''(X)(I_d \otimes X)) (I_{d^2} + K_{d,d})(I_d \otimes (X^TV))\mathrm{vec}(I_d).
\end{align*}
Using Lemma \ref{PropKronProd} (6), is possible to simplify this impression further. First, note that
\[
((X \cs \nabla f(X)) \otimes V)\mathrm{vec}(I_d) = \mathrm{vec}(V(X \cs \nabla f(X))).
\]
Moreover, we have
\[
(I_d \otimes (X^TV))\mathrm{vec}(I_d) = \mathrm{vec}(X^TV)
\]
and by \eqref{CommMat}, it follows that
\[
K_{d,d} \mathrm{vec}(X^TV) = \mathrm{vec}(V^TX),
\]
so that
\[
\frac{1}{2}(I_{d^2} + K_{d,d})(I_d \otimes (X^TV))\mathrm{vec}(I_d) = \mathrm{vec}(X \cs V).
\]
To continue,
\[
(I_d \otimes \nabla f(X)) \mathrm{vec}(X \cs V) = \mathrm{vec}(\nabla f(X)(X \cs V))
\]
as well as
\[
(I_d \otimes X)\mathrm{vec}(X \cs V) = \mathrm{vec}(X(X \cs V)).
\]
Putting everything together, we thus obtain
\begin{align*}
    \vec(\nabla \psi_V(X)) &= f''(X)\mathrm{vec}(V) - \mathrm{vec}(V(X \cs \nabla f(X)))\\
    &\quad - \mathrm{vec}(\nabla f(X)(X \cs V)) - f''(X)\mathrm{vec}(X(X \cs V)).
\end{align*}
Restricting this from $X \in \mathbb{R}^{n \times d}$ to $A \in W_{n,d}$ and recalling the projection $P_A$, we therefore have
\[
    \mathrm{vec}(\nabla \psi_V(A)) = f''(A)\mathrm{vec}(P_AV) - \mathrm{vec}(V(A \cs \nabla f(A)) + \nabla f(A)(A \cs V)).
\]

In terms of intrinsic derivatives (and now using the short-hand notation introduced in Proposition \ref{2ndOrDerSt}), we obtain
\[
\nabla_W \psi_V(A) = P_A f''(A) P_AV - P_A[V(A \cs \nabla f(A))] - P_A[\nabla f(A)(A \cs V)].
\]
Recall the matrix $B = f''(A) - (A \cs \nabla f(A)) \otimes I_n$ from \eqref{MatrixBSt}. By Proposition \ref{PropKronProd} (5),
\[
\tilde{B}\mathrm{vec}(V) = ((A \cs \nabla f(A)) \otimes I_n) \mathrm{vec}(V) = \mathrm{vec}(V(A \cs \nabla f(A))),
\]
so that if we write $B = f''(A) - \tilde{B}$, we obtain
\begin{align*}
\nabla_W \psi_V(A) &= P_A f''(A) P_AV - P_A\tilde{B}V - P_A[\nabla f(A)(A \cs V)]\\
&= P_A f_W''(A) P_AV - P_A\tilde{B}(V-P_AV) - P_A[\nabla f(A)(A \cs V)]
\end{align*}
in view of Proposition \ref{2ndOrDerSt}. As above, we note that 
\[
\tilde{B}V = (V-P_AV)(A \cs \nabla f(A)),
\]
so that altogether, we arrive at
\begin{align*}
f''_W(A)V &= \nabla_W\langle \nabla_W f(A), V \rangle + P_A[(V-P_A V)(A\cs\nabla f(A))]\\
&\quad+ P_A[\nabla f(A)(A \cs V)].
\end{align*}

To finish the proof, it remains to note that
\begin{align*}
&\quad P_A[(V-P_A V)(A\cs\nabla f(A))] + P_A[\nabla f(A)(A \cs V)] - P_A[\nabla_W f(A)(A \cs V)]\\
&= P_A[A(A\cs V)(A\cs\nabla f(A))] + P_A[A(A \cs \nabla f(A))(A\cs V)]\\
&= P_A[A((A\cs V)\cs(A\cs\nabla f(A)))] = 0,
\end{align*}
where the last step follows by an easy calculation using $A^TA = I_d$.
\end{proof}

\section{Second order modulus of gradient}\label{sec:2OrMoG}

Recall that the generalized second order modulus of the gradient on the Stiefel manifold $W_{n,d}$ is defined by
\begin{align*}
    |\nabla^{(2)}f(A)| &= |\nabla_W|\nabla_Wf(A)||\\
    &= \limsup_{A'\to A}\frac{\big||\nabla_Wf(A)|-|\nabla_Wf(A')|\big|}{|A-A'|}.
\end{align*}
Typically, explicitly calculating $|\nabla^{(2)} f(A)|$ is not easy, however, and therefore, motivated by the Euclidean or spherical calculus, we may hope for an estimate of the form $|\nabla^{(2)}f(A)| \le \lVert f''_W(A) \rVert_\mathrm{op}$. Indeed, we have the following result.

\begin{proposition}\label{Prop9.2.1St}
For any $\mathcal{C}^2$-smooth function $f$ on $W_{n,d}$, $|\nabla_W f|$ has finite Lipschitz semi-norm, and for any $A \in W_{n,d}$,
\[
|\nabla^{(2)}f(A)| = |\nabla_Wf(A)|^{-1}|f''_W(A)\nabla f_W(A)|,
\]
where the right-hand side has to be understood as $\lVert f''_W(A) \rVert_\mathrm{op}$ if $|\nabla_Wf(A)| = 0$. In particular, we always have
\[
|\nabla^{(2)}f(A)| \le \lVert f''_W(A) \rVert_\mathrm{op}.
\]
\end{proposition}

\begin{proof}
Let us first prove that the function $|\nabla_Wf(A)|$ has finite Lipschitz semi-norm. Since the first two intrinsic derivatives of $f$ are continuous and therefore bounded on the compact manifold $W_{n,d}$, we obtain from Proposition \ref{Prop9.1.3St} that
\[
|\nabla_W \langle \nabla_Wf(A), V \rangle| \le C
\]
for any $V \in \mathbb{R}^{n \times d}$ such that $|V| \equiv \lVert V \rVert_\mathrm{HS} = 1$, where $C$ is some constant independent of $A$ and $V$. Hence, the function $A \mapsto \langle\nabla_W f(A), V\rangle$ has Lipschitz semi-norm bounded by $C$, i.\,e.
\[
|\langle \nabla_Wf(A'), V \rangle - \langle \nabla_Wf(A), V\rangle| \le Cd(A,A')
\]
for all $A, A' \in W_{n,d}$. Therefore, taking the supremum over all $V$ and applying the triangle inequality yields
\[
||\nabla_Wf(A')| - |\nabla_Wf(A)||\le |\nabla_Wf(A') - \nabla_Wf(A)| \le Cd(A',A),
\]
which had to be proven.

To show the identity for the second order modulus of the gradient, fix $A \in W_{n,d}$. By the definition of the intrinsic gradient and Proposition \ref{Prop9.1.3St}, we have
\[
\langle \nabla_Wf(A'),V\rangle = \langle \nabla_Wf(A),V\rangle + \langle \tilde{V}, A'-A\rangle + o(|A'-A|),
\]
where
\[
\tilde{V} = f''_W(A)V - P_A(\nabla_W f(A)(A \cs V)).
\]
Moreover, a closer analysis (using the integral form of the Taylor formula and the compactness of $W_{n,d}$, which implies that every continuous function is already uniformly continuous) yields that the remainder term in the Taylor expansion can be bounded independently of $V \in \mathbb{R}^{n\times d}$ such that $|V| = 1$, i.\,e.
\[
\sup_{|V| = 1} |\langle \nabla_Wf(A'),V\rangle - \langle \nabla_Wf(A),V\rangle - \langle \tilde{V}, A'-A\rangle| \le \varepsilon(|A-A'|),
\]
where $\varepsilon(t)$ is some function which satisfies $\varepsilon(t) \to 0$ as $t \to 0$.

The next step is to rewrite the Taylor formula as
\begin{equation}\label{TayF}
\langle \nabla_Wf(A'),V\rangle = \langle \nabla_Wf(A) + L,V\rangle + o(|A'-A|).
\end{equation}
To this end, we first choose a suitable $\tilde{L}$ such that $\langle \tilde{V}, A'-A\rangle = \langle \tilde{L}, V \rangle$. Obviously,
\[
    \langle f_W''(A)V, A'-A \rangle = \langle f''_W(A)(A'-A), V \rangle.
\]
Moreover,
\begin{align*}
&\quad P_A(\nabla_Wf(A)(A \cs V))\\
&= \nabla_W f(A) (A \cs V) - A(A \cs (\nabla_W f(A)(A \cs V)))\\
&= \frac{1}{2} (\nabla_W f(A)A^TV + \nabla_W f(A)V^TA) - \frac{1}{4} (AA^T\nabla_W f(A)A^TV\\
&\qquad+ AA^T\nabla_W f(A)V^TA + AA^TV\nabla_W f(A)^T A + A V^TA \nabla_W f(A)^T A).
\end{align*}
Recall that $\langle U, V \rangle = \mathrm{tr}(U^TV)$ and that the trace is invariant under cyclic permutations. Therefore, we may easily verify the general rules
\begin{equation}\label{genrul}
    \langle UVW,X \rangle = \langle U^TXW^T, V \rangle,\qquad \langle UV^TW, X \rangle = \langle WX^TU, V \rangle,
\end{equation}
where $V, X \in \mathbb{R}^{n \times d}$ and $U \in \mathbb{R}^{n \times n}$, $W \in \mathbb{R}^{d \times d}$ in the first identity or $U, W \in \mathbb{R}^{n \times d}$ in the second one. We now apply this inequality to all six terms appearing in $P_A(\nabla_Wf(A)(A\cs V))$ with $V$ as above and $X = A'-A$. For instance, for the first term this yields
\[
\langle \nabla_Wf(A)A^TV, A'-A \rangle = \langle A \nabla_Wf(A)^T(A'-A), V \rangle.
\]
Proceeding similarly, we arrive at
\[
\langle \nabla_W f(A)(A \cs V), A'-A \rangle = \langle A(\nabla_W f(A) \cs (A'-A)), V \rangle
\]
as well as
\[
\langle A(A \cs (\nabla_W f(A)(A \cs V))), A'-A \rangle = \langle A((A^T\nabla_W f(A)) \cs (A \cs (A'-A))), V \rangle.
\]
Altogether, we obtain
\[
    \tilde{L} = f_W''(A)(A'-A) - A[\nabla_W f(A) \cs (A'-A)
     - (A^T\nabla_W f(A)) \cs (A \cs (A'-A))].
\]
Now we define
\[
L := f_W''(A)(A'-A) - A(\nabla_W f(A) \cs (A'-A))
\]
To see that \eqref{TayF} holds, it remains to show that $\tilde{L} - L = o(|A'-A|)$ as $A' \to A$, $A, A' \in W_{n,d}$. To this end, note that by an easy calculation,
\[
A \cs (A'-A) = \frac{1}{2} (A'-A)^T(A-A') = o(|A'-A|).
\]
From here, the claim immediately follows by compactness arguments.

Now we take an absolute value on both sides of the Taylor formula \eqref{TayF} and take the supremum over all $V$ such that $\mathrm{vec}(V) \in S^{nd-1}$. This leads to
\[
|\nabla_Wf(A')| = |\nabla_Wf(A) + L| + o(|A'-A|).
\]
Next, we write
\[
|\nabla_Wf(A) + L|^2 = |\nabla_Wf(A)|^2 + 2 \langle \nabla_Wf(A), L \rangle + |L|^2.
\]
Noting that
\[
\langle \nabla_W f(A), A(A\cs \nabla f(A))\rangle = \mathrm{tr} (\nabla_W f(A)^T A(A\cs \nabla f(A))) = 0
\]
since $\nabla_W f(A)^T A$ is antisymmetric (as $\nabla_W f(A) \in T_A$) and $A \cs \nabla f(A)$ is symmetric, we obtain
\[
\langle \nabla_W f(A), L \rangle = \langle \nabla_Wf(A), f_W''(A)(A'-A) \rangle = \langle U, A'-A \rangle,
\]
where
\[
U := f_W''(A)\nabla_Wf(A) \equiv f_W''(A)\mathrm{vec}(\nabla_Wf(A)).
\]
Since $|L|^2 = O(|A'-A|^2)$, we obtain
\[
|\nabla_Wf(A) + L|^2 = |\nabla_Wf(A)|^2 + 2 \langle U, A'-A \rangle + o(|A'-A|).
\]

If $|\nabla_Wf(A)| > 0$, it therefore follows that
\[
|\nabla_Wf(A) + L| = |\nabla_Wf(A)| + |\nabla_Wf(A)|^{-1} \langle U, A'-A \rangle + o(|A'-A|).
\]
Hence,
\[
|\nabla_Wf(A')| - |\nabla_Wf(A)| = |\nabla_Wf(A)|^{-1} \langle U, A'-A \rangle + o(|A'-A|)
\]
and thus
\begin{align*}
\limsup_{A'\to A} \frac{\big||\nabla_Wf(A')| - |\nabla_Wf(A)|\big|}{|A'-A|}
&= |\nabla_Wf(A)|^{-1} \limsup_{A' \to A} \frac{\big|\langle U, A'-A \rangle\big|}{|A'-A|}\\
&= |\nabla_Wf(A)|^{-1} |\nabla_W\psi_U(A)|,
\end{align*}
where $\psi_U (A) := \langle U, A \rangle$. As noted after Proposition \ref{2ndOrDerSt}, $U \in T_A$, so that $\nabla_W \psi_U(A) = U$. Thus, we arrive at
\[
|\nabla^{(2)}_W(A)| = |\nabla_Wf(A)|^{-1} |f_W''(A)\nabla_Wf(A)|
\]
if $|\nabla_Wf(A)| > 0$.

It remains to consider the case where $|\nabla_Wf(A)| = 0$. Here, $L = f_W''(A)(A'-A)$, and the Taylor formula reads
\[
|\nabla_W f(A')| = |L| + o(|A'-A|).
\]
It follows that
\begin{align*}
    |\nabla^{(2)}_Wf(A)| &= \limsup_{A' \to A} \frac{|\nabla_Wf(A')|}{|A'-A|}
    = \limsup_{A' \to A} \frac{|f_W''(A)(A'-A)|}{|A'-A|}\\
    &= \limsup_{V \to 0, V \in T_A^\perp} \frac{|f_W''(A)V|}{|V|} = \lVert f_W''(A) \rVert_\mathrm{op},
\end{align*}
which finishes the proof.
\end{proof}

\section{Derivatives on Grassmann manifolds}\label{sec:DerGr}

Let us adapt the results of the previous two sections to functions on Grassmann manifolds. Much of what follows relies on similar arguments as in the Stiefel case, and for this reason we will often only sketch the arguments. To introduce a notion of differentiability on $G_{n,d}$, first recall that the tangent space in $P \in G_{n,d}$ is given by
\[
T_P := \{S \in \Rns \colon S = SP + PS\} \equiv \{S \in \Rns \colon S = [[S,P],P]\},
\]
where for any $M,N \in \mathbb{R}^{n \times n}$, $[M,N] = MN - NM$ denotes the matrix commutator. A function $f \colon G_{n,d} \to \mathbb{R}$ is differentiable at $X \in G_{n,d}$ if it admits a Taylor expansion
\begin{equation}\label{defderGr}
f(P') = f(P) + \langle M, P' - P\rangle + o\big(|X' - X|\big) \quad \mathrm{as} \ \ P' \rightarrow P, \ \ P' \in G_{n,d}
\end{equation}
with some $M \in \Rns$. Among all such $M$, there exists a unique $M_0$ of smallest (Euclidean) length, called the intrinsic (first) derivative or gradient of $f$ at $P$ and denoted $\nabla_G f(P)$. The length of $\nabla_G f(P)$ agrees with the generalized modulus of the gradient \eqref{genmod} applied to the Grassmann manifold. In passing, note that unlike in case of the Stiefel manifold, for Grassmann manifolds the Euclidean metric and the canonical metric lead to the same notion of differentiablity. As in case of the Stiefel manifold, the minimization problem \eqref{defderGr} translates into
\[
\lVert M - B \rVert \to \mathrm{min}\qquad \text{over all $B \in T_P^\perp$,}
\]
which is solved uniquely for the orthogonal projection $M$ onto $T_P$.

If we consider functions which are defined and smooth in an open neighborhood of $G_{n,d}$ in the ambient space, e.\,g.\ $\Rns$, we may take the Euclidean gradient $\nabla f(P) = \mathrm{mat}(Df(P)^T)$ (cf.\ Section \ref{sec:Nabf}) and project it onto $T_P$, which gives back $\nabla_G f(P)$. As the projection $\pi_P \colon \Rns \to T_P$ is given by
\begin{equation}\label{projfGr}
\pi_P M := [P,[P,M]] = PM + MP - 2PMP,
\end{equation}
we immediately arrive at the following analogue of Proposition \ref{projSt}.

\begin{proposition}\label{projGr}
Let $f$ be defined and $\mathcal{C}^1$-smooth in some open neighborhood of $G_{n,d}$ in $\Rns$. Then, the intrinsic first derivative of $f$ at $P \in G_{n,d}$ is given by the projection onto $T_P$
\[
\nabla_G f(P) = \pi_P\nabla f(P) = [P,[P,\nabla f(P)]].
\]
In particular, $|\nabla_G f(P)| \leq |\nabla f(P)|$ for any $P \in G_{n,d}$.
\end{proposition}

In fact, sometimes yet a further embedding might be convenient, so that we regard $G_{n,d}$ as a submanifold of $\mathbb{R}^{n \times n}$ and take the Euclidean derivatives of some extension of $f$ to an open neighborhood in $\mathbb{R}^{n \times n}$. However, in this case, we may project $\nabla f(P)$ onto the tangent space of $\Rns$ (which equals $\Rns$) by applying the projection $\pis \colon \mathbb{R}^{n \times n} \to \Rns$ given by $\pis(M) := (M + M^T)/2$. Then, we may proceed as in Proposition \ref{projGr} for $\pis(\nabla f(P))$.

For any $C^2$-smooth function $f$ on $G_{n,d}$ at a given point $P \in W_{n,d}$, the intrinsic second order derivative is the matrix $B \in \mathbb{R}^{n^2 \times n^2}$ of smallest Hilbert--Schmidt norm which satisfies
\begin{equation}\label{Taylor2ndOrderGr}
f(P') = f(P) + \langle\nabla_G f(P), P'-P\rangle + \frac{1}{2} \langle B(P'-P), P'-P \rangle + o\big(|P' - P|^2\big)
\end{equation}
as $P' \rightarrow P$, $P' \in G_{n,d}$. It will be denoted $f_G''(X)$. Considering functions $f$ which are defined and smooth in some open neighborhood of $G_{n,d}$ in $\Rns$, it is possible to express the intrinsic second order derivative in terms of the Euclidean derivatives of $f$. Instead of providing the details, we refer to \cite[Theorem 2.4]{HHT07} where corresponding calculations have been done. Adapting them to our framework yields the following analogue of Proposition \ref{2ndOrDerSt}.

\begin{proposition}\label{2ndOrDerGr}
The intrinsic second derivative of $f$ at each $P \in G_{n,d}$ is given by the relation
\[
f_G''(P)V = \pi_P f''(P) \pi_PV - [P,[\nabla f(P),\pi_PV]]
\]
for any $V \in \Rns$. Here, the expression on the right hand side has to be understood as
\[
\vec(\pi_P \mathrm{mat}[f''(P)\mathrm{vec}(\pi_PV)]) - \vec([P,[\nabla f(P),\pi_PV]]).
\]
\end{proposition}

Note that $[P,[\nabla f(P),\pi_PV]] \in T_P$. In particular, for all $P \in G_{n,d}$ and $V \in \Rns$, $f_G''(P)V \in T_P$, $f''_G(P)P = 0$, and hence $\langle f_G''(P)V, P \rangle = 0$ for all $V \in \Rns$. Furthermore, we have the contraction property
\[
\|f_G''(P)\|_\mathrm{HS} \leq \|f''(P) - [P,[\nabla f(P),\pi_PV]]\|_{{\rm HS}},
\]
which also holds for the operator norm. Similarly to the first order case, sometimes it is necessary to extend $f$ to a smooth function on some neighborhood in $\mathbb{R}^{n \times n}$. In this case, we may take the usual Euclidean Hessian $f''(P)$, from which we get back the Hessian on $\Rns$ by considering $\pis f''(P) \pis$.

Moreover, we have the following analogue of Proposition \ref{Prop9.1.3St}.

\begin{proposition}\label{Prop9.1.3Gr}
Given a $C^2$-smooth function $f$ on $G_{n,d}$, for all $P \in G_{n,d}$ and $V \in \Rns$, we have
\[
f_G''(P)V = \nabla_G\psi_V(P) - [P,[\nabla_G f(P), V]].
\]
Here, the left-hand side has to be read as $f''(W)\mathrm{vec}(V)$ and the right-hand side has to be understood as vectorized.
\end{proposition}

In particular, if $V \in T_P$ it follows from \cite[Lemma 2.2]{HHT07} that
\[
f''_G(P)V = \nabla_G\langle \nabla_G f(P), V \rangle.
\]

Finally, we also have an analogue of Proposition \ref{Prop9.2.1St}.

\begin{proposition}\label{Prop9.2.1Gr}
For any $\mathcal{C}^2$-smooth function $f$ on $G_{n,d}$, $|\nabla_G f|$ has finite Lipschitz semi-norm, and for any $P \in G_{n,d}$,
\[
|\nabla^{(2)}_Gf(P)| = |\nabla_Gf(P)|^{-1}|f''_G(P)\nabla f_G(P)|,
\]
where the right-hand side has to be understood as $\lVert f''_G(P) \rVert_\mathrm{op}$ if $|\nabla_Gf(P)| = 0$. In particular, we always have
\[
|\nabla^{(2)}_Gf(P)| \le \lVert f''_G(P) \rVert_\mathrm{op}.
\]
\end{proposition}

As the proofs of Proposition \ref{Prop9.1.3Gr} and Proposition \ref{Prop9.2.1Gr} are mostly an adaption of the arguments known from the Stiefel case, we defer them to the appendix.

\section{Proofs}\label{sec:Sob}

To prepare the proofs of our main results, we briefly revisit the representations of the Stiefel and Grassmann manifolds we use in this paper. In particular, we shall discuss the Lipschitz properties of the map $\pist \colon W_{n,d} \to G_{n,d}$, $A \mapsto AA^T \equiv P_A$. Clearly, $\pist$ is $2\sqrt{d}$-Lipschitz as
\[
\lVert AA^T - A'A'^T \rVert_\mathrm{HS} \le \lVert A(A^T - A'^T) \rVert_\mathrm{HS} + \lVert (A-A')A'^T \rVert_\mathrm{HS} \le 2\sqrt{d} \lVert A-A' \rVert_\mathrm{HS}
\]
since $\lVert A \rVert_\mathrm{HS}, \lVert A' \rVert_\mathrm{HS} = \sqrt{d}$. However, this Lipschitz constant is too weak for our purposes. To avoid the dependency on $d$, more subtle arguments are needed.

Recall that if $M_1, M_2 \subset \mathbb{R}^n$ are two subspaces of dimension $d$, the principal angles $\theta_j$, $j = 1, \ldots, d$, between $M_1$ and $M_2$ are recursively defined by
\[
\cos \theta_j = \max_{u \in M_1} \max_{v \in M_2} \lvert \langle u, v \rangle \rvert = \langle u_j, v_j \rangle
\]
subject to the constraints
\[
\langle u_i, u \rangle = 0,\quad \langle v_i, v \rangle = 0,\quad i = 1, 2, \ldots, j-1.
\]
The vectors $\{u_1, \ldots, u_d\}$, $\{v_1, \ldots, v_d\}$ are called the principal vectors of the pair $(M_1,M_2)$. The concept of principal angles goes back to Jordan. Here we mainly follow the survey article \cite{Ga08}.

Note that the principal angles are uniquely defined and satisfy $0 \le \theta_1 \le \ldots \le \theta_d \le \pi/2$, while the principal vectors are not unique. However, they form orthonormal $d$-frames and can thus be interpreted as elements $U,V \in W_{n,d}$. Moreover, we have $\langle u_i, v_j \rangle = \delta_{ij} \cos\theta_i$. By \cite{Af57}, the respective orthogonal projections $P_U = \pist(U)$ and $P_V = \pist(V)$ satisfy
\[
P_UP_Vu_j = (\cos^2\theta_j)u_j,\quad P_VP_U v_j = (\cos^2\theta_j)v_j,\quad j = 1, \ldots, d,
\]
i.\,e., the non-zero eigenvalues of the matrices $P_UP_V$ and $P_VP_U$ are $\cos^2\theta_j$, $j = 1, \ldots, d$. This relation extends to all $A, B \in W_{n,d}$ such that $A = UO$ and $B = VO'$ for some $O,O' \in O(d)$ (since $P_A = P_U$ and $P_B = P_V$). Note also that by definition of the principal angles, for any $A,A' \in W_{n,d}$
\[
\max_{O \in O(d)} \langle A, AO' \rangle = \langle A,A'' \rangle = \tr(A^TA'') = \sum_{j=1}^d \cos\theta_j,
\]
where $\theta_j$ are the principal angles between the two subspaces induced by $A$ and $A'$ and $A''$ is the Stiefel matrix maximizing the term on the left-hand side.

\begin{lemma}\label{pistLip}
The map $\pist \colon W_{n,d} \to G_{n,d}$, $A \mapsto AA^T = P_A$ is $\sqrt{2}$-Lipschitz.
\end{lemma}

\begin{proof}
Using principle angles and the notation introduced above, we have
\begin{align*}
\lVert P_A-P_{A'} \rVert_\mathrm{HS}^2 &=  2d - 2 \tr (P_A P_{A'}) = 2d - 2\sum_{j=1}^d \cos^2\theta_j=2\sum_{j=1}^d(1-\cos^2\theta_j)\\
&\le 2\sum_{j=1}^d 2(1- \cos\theta_j) = 2(2d - 2 \tr(A^TA''))\\
&\le 2(2d - 2 \tr(A^TA')) = 2||A-A'||_\mathrm{HS}^2,
\end{align*}
where in the first inequality we used that $(1-x^2)/(1-x) \le 2$ for any $x \in [0,1]$.
\end{proof}

One may wonder whether the map $P_A$ might even be $1$-Lipschitz. However, simple examples are sufficient to show that this cannot be true (e.\,g., consider $d=1$ and the vectors $A = (1,0, \ldots, 0)^T$ and $A' = (1/\sqrt{n}, \ldots, 1/\sqrt{n})^T$).

The core of our arguments is a logarithmic Sobolev inequality for Stiefel and Grassmann manifolds. Even if we are not aware of a source where log-Sobolev inequalities for Stiefel and Grassmann manifolds are rigorously formulated, they may easily be derived by a simple projection argument. We emphasize that these inequalities are designed for the representations of $W_{n,d}$ and $G_{n,d}$ we use all over this paper.

\begin{proposition}\label{LSU}
\begin{enumerate}
\item For any $d < n$, $W_{n,d}$ satisfies a logarithmic Sobolev inequality with constant $4/(n-2)$, i.\,e.\ for any $f \colon W_{n,d} \to \mathbb{R}$ sufficiently smooth,
\[
\mathrm{Ent}_{\mu_{n,d}} (f^2) \le \frac{8}{n-2} \int_{W_{n,d}} |\nabla_W f|^2 d\mu_{n,d}.
\]
\item For any $d < n$, $G_{n,d}$ satisfies a logarithmic Sobolev inequality with constant $8/(n-2)$, i.\,e.\ for any $f \colon G_{n,d} \to \mathbb{R}$ sufficiently smooth,
\[
\mathrm{Ent}_{\nu_{n,d}} (f^2) \le \frac{16}{n-2} \int_{G_{n,d}} |\nabla_G f|^2 d\nu_{n,d}.
\]
\end{enumerate}
\end{proposition}

Let us briefly mention that for $d=1$ (i.\,e.\ the sphere), the optimal Sobolev constant is known to be $1/(n-1)$ as shown in \cite{MW82}. In particular, even if Proposition \ref{LSU} is not fully accurate, the behavior of the Sobolev constant for large values of $n$ agrees with the optimal result in case of the sphere. Moreover, note that if $d=n$, $W_{n,n} = O(n)$ has two connected components, which in particular implies that a log-Sobolev inequality cannot hold.

\begin{proof}[Proof of Proposition \ref{LSU}]
First recall that if $d < n$, we have
\[
W_{n,d} \cong SO(n)/SO(n-d).
\]
Indeed, identifying any matrix in $SO(n)$ with its first $n-1$ columns $e_1, \ldots, e_{n-1}$, this follows readily using the projection map $\varphi \colon SO(n) \to W_{n,d}$ which is given by $\varphi(e_1, \ldots, e_{n-1}) := (e_1, \ldots, e_d)$. By \cite[Theorem 5.16]{Me19}, the special orthogonal group $SO(n)$ equipped with the uniform probability measure and Hilbert--Schmidt metric satisfies a log-Sobolev inequality with constant $4/(n-2)$. Noting that the map $\varphi$ is $1$-Lipschitz, (1) therefore follows immediately.

To see (2), it remains to note that $(G_{n,d}, \nu_{n,d})$ is the the pushforward of $(W_{n,d},\mu_{n,d})$ under $\pist$, which is $\sqrt{2}$-Lipschitz according to Lemma \ref{pistLip} (so that the Sobolev constant is doubled).
\end{proof}

In the sequel, we will give the proofs of the concentration bound for function on Stiefel manifolds, i.\,e., Theorem \ref{Conc2ndOr} and Theorem \ref{ConckthOr}. As already shown in \cite{AS94}, logarithmic Sobolev inequalities imply certain $L^p$ norm inequalities. In the situation under consideration, for any locally Lipschitz function $g \colon W_{n,d} \to \mathbb{R}$ and for any $p \ge 2$,
\begin{equation}\label{LpNorms}
    \lVert g \rVert_p^2 \le \lVert g \rVert_2^2 + 4(n-2)^{-1} (p-2)\lVert \nabla_W g \rVert_p^2.
\end{equation}
Using Proposition \ref{projSt}, we moreover have
\begin{equation}\label{LpNormsEucl}
    \lVert g \rVert_p^2 \le \lVert g \rVert_2^2 + 4(n-2)^{-1} (p-2)\lVert \nabla g \rVert_p^2.
\end{equation}
Let us also recall that since $\mu_{n,d}$ satisfies a log-Sobolev inequality, it also satisfies a Poincar\'{e} inequality with the same constant, i.\,e.
\begin{equation}\label{PoincareIneq}
    \mathrm{Var}_{\mu_{n,d}}(f) \equiv \int_{W_{n,d}} (f-\mu_{n,d}(f))^2 d\mu_{n,d} \le \frac{4}{n-2} \int_{W_{n,d}} |\nabla_W f|^2 d\mu_{n,d}
\end{equation}
for all $f \colon W_{n,d} \to \mathbb{R}$ sufficiently smooth.

We now first prove Theorem \ref{ConckthOr}. Note that its proof follow the lines of the proofs established in \cite{BGS19,GSS21b}.

\begin{proof}[Proof of Theorem \ref{ConckthOr}]
Applying \eqref{LpNormsEucl} to $g = \lVert f^{(\ell)} \rVert_\mathrm{op}$ and recalling that for Euclidean derivatives, $|\nabla\lVert f^{(\ell)} \rVert_\mathrm{op}| \le \lVert f^{(\ell + 1)} \rVert_\mathrm{op}$ (cf. \cite[Lemma 4.1]{BGS19}), we obtain that
\[
\lVert f^{(\ell)} \rVert_{\mathrm{op},p}^2 \le \lVert f^{(\ell)} \rVert_{\mathrm{op},2}^2 + 4(n-2)^{-1} (p-2)\lVert \nabla f^{(\ell + 1)} \rVert_{\mathrm{op},p}^2
\]
for any $\ell = 1, \ldots, k-1$, which yields
\begin{align}\label{rec}
\lVert f \rVert_p^2 &\le \lVert f \rVert_2^2 + \sum_{\ell = 1}^{k-1} \Big(\frac{4(p-2)}{n-2}\Big)^\ell \lVert f^{(\ell)} \rVert_{\mathrm{op},2}^2 + \Big(\frac{4(p-2)}{n-2}\Big)^k \lVert f^{(k)} \rVert_{\mathrm{op},\infty}^2\notag\\
&\le \frac{4}{n-2} \lVert f^{(1)} \rVert_2^2 + \sum_{\ell = 1}^{k-1} \Big(\frac{4(p-2)}{n-2}\Big)^\ell \lVert f^{(\ell)} \rVert_{\mathrm{op},2}^2 + \Big(\frac{4(p-2)}{n-2}\Big)^k \lVert f^{(k)} \rVert_{\mathrm{op},\infty}^2\notag\\
&\le \sum_{\ell = 1}^{k-1} \Big(\frac{4(p-1)}{n-2}\Big)^\ell \lVert f^{(\ell)} \rVert_{\mathrm{op},2}^2 + \Big(\frac{4(p-1)}{n-2}\Big)^k \lVert f^{(k)} \rVert_{\mathrm{op},\infty}^2
\end{align}
for any $p \ge 2$, where the second step follows by Poincar\'{e} inequality.

To see (1), plugging in the assumptions we arrive at
\[
\lVert f \rVert_p^2 \le \Big(\frac{4}{n-2}\Big)^k \sum_{\ell=1}^k p^\ell \le \frac{1}{1 - p^{-1}} \Big(\frac{4p}{n-2}\Big)^k \le \Big(\frac{8p}{n-2}\Big)^k
\]
and therefore $\lVert f \rVert_p \le (8p/(n-2))^{k/2}$ for any $p \ge 2$. If $p \le 2$, we may estimate $\lVert f \rVert_p \le \lVert f \rVert_2 \le (16/(n-2))^{k/2}$. In particular, we obtain that for any $m \ge 1$,
\[
\lVert |f|^{2/k} \rVert_m = \lVert f \rVert_{2m/k}^{2/k} \le \gamma m
\]
with $\gamma = 16/(n-2)$. In particular, by an easy calculation (cf. \cite[Eq.\ (2.17)]{BGS19}, this yields that $\int_{W_{n,d}} \exp(c'|f|) d\mu_{n,d} \le 2$ for $c' = 1/(2\gamma e)$, i.\,e.\ (1).

To show (2), taking roots in \eqref{rec} yields
\[
\lVert f \rVert_p \le \sum_{\ell = 1}^{k-1} \Big(\frac{4(p-1)}{n-2}\Big)^{\ell/2} \lVert f^{(\ell)} \rVert_{\mathrm{op},2} + \Big(\frac{4(p-1)}{n-2}\Big)^{k/2} \lVert f^{(k)} \rVert_{\mathrm{op},\infty}.
\]
Now, \cite[Proposition 4]{GSS21b} yields
\[
\mu_{n,d}(|f-\mu_{n,d}(f)| \ge t) \le 2 \exp\Big(-\frac{1}{C}\min\Big(\min_{\ell = 1, \ldots, k-1}\frac{t^{2/\ell}}{\lVert f^{(\ell)} \rVert_{\mathrm{op},2}^{2/\ell}}, \frac{t^{2/k}}{\lVert f^{(k)} \rVert_{\mathrm{op},\infty}^{2/k}}\Big)\Big),
\]
where we may choose
\[
C = \frac{4e^2k^2}{\log(2)(n-2)},
\]
which completes the proof.
\end{proof}

For the proof of Theorem \ref{Conc2ndOr}, we moreover need the following lemma which relates the $L^2$ norms of the derivatives of first and second order under a centering assumption.

\begin{lemma}\label{1to2}
Let $f \colon W_{n,d} \to \mathbb{R}$ be a $\mathcal{C}^2$-smooth function which has centered first order intrinsic derivatives, i.\,e.\ $\mu_{n,d}(\nabla_W f) = 0$. Then,
\[
\int_{W_{n,d}} |\nabla_W f|^2 d\mu_{n,d} \le \frac{8}{n-2-8d} \int_{W_{n,d}} \lVert f''_W \rVert_\mathrm{HS}^2 d\mu_{n,d}.
\]
\end{lemma}

\begin{proof}
Write
\[
|\nabla_W f(A)|^2 = \sum_{i,j} \langle \nabla_W f(A), e_{ij} \rangle^2,
\]
where $i \le n$, $j \le d$ and $e_{ij}$ is the ``indicator matrix'' whose $(i,j)$-th entry equals $1$ while all other entries are zero. By assumption, the functions $A \mapsto \langle \nabla_W f(A), e_{ij} \rangle$ are centered. Therefore, applying the Poincar\'{e} inequality \eqref{PoincareIneq} yields
\[
\frac{n-2}{4} \int_{W_{n,d}} \langle \nabla_W f(A), e_{ij} \rangle^2 d \mu_{n,d} \le \int_{W_{n,d}} |\nabla_W (\nabla_W f(A), e_{ij} \rangle)|^2 d\mu_{n,d}.
\]
By Proposition \ref{Prop9.1.3St}, we have
\[
\nabla_W (\nabla_W f(A), e_{ij} \rangle) = f''_W(A)e_{ij} - P_A(\nabla_Wf(A)(A\cs e_{ij})),
\]
so that we may estimate
\begin{align*}
    &\quad \int_{W_{n,d}} |\nabla_W (\nabla_W f(A), e_{ij} \rangle)|^2 d\mu_{n,d}\\
    &\le 2\int_{W_{n,d}} |f''_W(A)e_{ij}|^2 d\mu_{n,d} + 2\int_{W_{n,d}} |P_A(\nabla_Wf(A)(A\cs e_{ij}))|^2 d\mu_{n,d}.
\end{align*}
By contractivity and using $|A_{ij}|^2 \le 1$,
\[
|P_A(\nabla_Wf(A)(A\cs e_{ij}))|^2 \le |(\nabla_Wf(A)(A\cs e_{ij})|^2 \le d |\nabla_Wf(A)|^2.
\]
Plugging in, summing up and noting that
\[
\sum_{i,j} |f_W''(A)e_{ij}|^2 = \lVert f_W''(A) \rVert_\mathrm{HS}^2
\]
(recall that $i \le n$, $j \le d$ and $f''_W(A)$ is an $nd \times nd$ matrix) completes the proof.
\end{proof}

\begin{proof}[Proof of Theorem \ref{Conc2ndOr}]
Obviously, we may argue as in the proof of Theorem \ref{ConckthOr} with \eqref{LpNormsEucl} replaced by \eqref{LpNorms} and \cite[Lemma 4.1]{BGS19} replaced by Proposition \ref{Prop9.2.1St} (thus involving intrinsic derivatives) to obtain
\[
\lVert f \rVert_p^2 \le \frac{4(p-1)}{n-2}\lVert \nabla_W f \rVert_2^2 + \Big(\frac{4(p-1)}{n-2}\Big)^2\lVert f''_W \rVert_{\mathrm{op},\infty}^2,
\]
which is the ``intrinsic analogue'' of \eqref{rec} for $k=2$. Therefore, to see (1) and (2), we may just copy the proof of Theorem \ref{ConckthOr}. Moreover, (3) follows by applying Lemma \ref{1to2}.
\end{proof}

Finally, to address Theorem \ref{Conc2ndOrGr} and Theorem \ref{ConckthOrGr}, note that for any locally Lipschitz function $g \colon G_{n,d} \to \mathbb{R}$ and for any $p \ge 2$,
\[
    \lVert g \rVert_p^2 \le \lVert g \rVert_2^2 + 8(n-2)^{-1} (p-2)\lVert \nabla_WGg \rVert_p^2,
\]
which is an analogue of \eqref{LpNorms}. From here, we may essentially copy the proofs of Theorem \ref{Conc2ndOr} and Theorem \ref{ConckthOr} with the obvious adaptions (which, in particular, lead to slightly different constants).

\section{Concentration results for polynomial chaos}\label{Sec:PolCh}

Arguably the most classical object in the study of higher order concentration is polynomial chaos. In the context of Stiefel manifolds, by a polynomial chaos of order $k$ we refer to functionals of the form
\[
f_k(A) := \sum_{i_1,j_1=1}^{n,d} \cdots \sum_{i_k,j_k=1}^{n,d} c_{i_1j_1, \ldots, i_kj_k} A_{i_1j_1} \cdots A_{i_kj_k},
\]
where $c_{i_1j_1, \ldots, i_kj_k}$ are real-valued coefficients. For $k=1,2$, these are linear or quadratic forms, respectively. Functionals of this type are for instance related to angles and distances between such random subspaces like $\det(A^TC)$, where $A \in W_{n,d}$ and $C \in \mathbb{R}^{n \times d}$ is a matrix of coefficients which might be random with, say, independent entries. (In this case, $\det(A^TC)$ is the scalar product of the exterior $d$-forms induced by $A$ and $C$ representing subspaces of dimension $d$.) In this section, we will mainly be interested in the case of $k=2$ but also provide some basic comments on other values of $k$.

In particular, information about the (mixed) moments of the Stiefel matrix entries $A_{ij}$ of order up to $k$ is needed. To this end, note that the $A_{ij}$ are identically distributed uncorrelated random variables with mean zero and variance $1/n$. Indeed, using the invariance under orthogonal transformations we clearly have $A_{ij} \sim A_{k\ell}$ and also, multiplying suitable columns or rows by $-1$,
\[
\int_{W_{n,d}} A_{ij}A_{k\ell} d \mu_{n,d} = - \int_{W_{n,d}} A_{ij}A_{k\ell} d\mu_{n,d} = 0
\]
for any $(i,j) \ne (k,\ell)$. Moreover, the $A_{ij}$ are centered since $A \sim -A$, and we have
\[
d = \sum_{k,\ell=1}^{n,d} A_{k\ell}^2 = \int_{W_{n,d}} \sum_{k,\ell=1}^{n,d} A_{k\ell}^2 d \mu_{n,d} = nd \int_{W_{n,d}} A_{ij}^2 d \mu_{n,d}
\]
for any $(i,j)$. Some of these relations may also deduced from the fact that the marginal distribution of each column of $A$ is the uniform distribution on the sphere.

These results also imply that the entries of a Grassmannian (projection) matrix $P$ satisfy
\[
\int_{G_{n,d}} P_{ij} d\nu_{n,d} = \begin{cases} d/n, & i=j\\ 0, & i \ne j \end{cases}.
\]
To see this, it suffices to note that writing $P = P_A = AA^T$, we have $P_{ij} = \sum_{k=1}^d A_{ik}A_{jk}$.

\subsection{First order results} To start, let us briefly study linear forms on $W_{n,d}$, which we may rewrite as
\[
f_1(A) = \langle V, A \rangle
\]
for some matrix of coefficients $V \in \mathbb{R}^{n \times d}$. Since $\nabla_W f_1(A) = P_A(V)$ and the entries of $A$ are centered, \eqref{LipschitzAbl} immediately yields
\begin{equation}\label{LinF}
\mu_{n,d}(|f_1| \ge t) \le 2 \exp\Big(-\frac{(n-1)t^2}{8\lVert P_A(V)\rVert_{\mathrm{HS}, \infty}^2}\Big).
\end{equation}
One may also choose the vector of coefficients $V$ at random according to some distribution with independent entries or dependent entries with higher order uncorrelatedness conditions. In this case, results like \eqref{LinF} hold conditionally given $V$, for instance. Functionals of this type may be regarded as natural generalizations of weighted sums $\langle X,\theta\rangle = \theta_1 X_1 + \dots + \theta_n X_n$ with uniformly distributed weights $\theta =(\theta_1, \ldots, \theta_n) \in S^{n-1}$ and some random vector $X$ with higher order uncorrelated entries (cf.\ \cite{BCG18,BCG20}), replacing the sphere by the Stiefel manifold.

\subsection{Second order results} The central object of this section are quadratic forms on $W_{n,d}$. Recall that a classical (second order) concentration result for quadratic forms in independent subgaussian random variables is the famous Hanson--Wright inequality, cf.\ \cite{HW71}. It states that for a random vector $X = (X_1, \ldots, X_n)$ with independent centered components with variance $1$,
\[
\mathbb{P}(|X^TMX - \mathrm{tr}(M)| \ge t) \le 2 \exp\Big(- \frac{1}{C} \min\Big(\frac{t^2}{\lVert M \rVert_\mathrm{HS}^2}, \frac{t}{\lVert M \rVert_\mathrm{op}}\Big)\Big)
\]
for any $t \ge 0$, where the constant $C > 0$ depends on the subgaussian norm of the coordinates $X_i$.

We derive analogues of the Hanson--Wright inequality for Stiefel manifolds, i.\,e., we consider
\begin{equation}\label{notquf}
f_2(A) = \mathrm{vec}(A)^TM\mathrm{vec}(A) \equiv A^TMA
\end{equation}
for some matrix $M \in \mathbb{R}^{nd \times nd}$ which, for simplicity, we assume to be symmetric. Note that the informal notation $A^TMA$ extends \eqref{shn}. For instance, to go back to the example sketched at the beginning of this section, we may consider $\det(A^TC)$, where for simplicity we assume $C \in \mathbb{R}^{n \times d}$ to be fixed, in the case of $d=2$. Here,
\[
\det(A^TC) = \Big(\sum_{i=1}^n a_{i1}c_{i1}\Big)\Big(\sum_{j=1}^n a_{i2}c_{i2}\Big) - \Big(\sum_{i=1}^n a_{i1}c_{i2}\Big)\Big(\sum_{j=1}^n a_{i2}c_{i1}\Big) = A^TMA
\]
for the matrix $M$ with entries $M_{ij,kl} = (c_{ij}c_{kl} - c_{il}c_{kj})/2$. More advanced examples will be discussed in a further subsection.

In the next theorem, we present three different Hanson--Wright type inequalities for Stiefel manifolds. Write
\[
U := \Big(\sum_{k,l}M_{ij,kl}A_{kl}\Big)_{i,j},\qquad B := M - (A \cs U) \otimes I_n
\]
for $i, k \le n$, $j, l \le d$, noting that $\nabla_W f_2(A) = 2P_AU$ and $f''_{2W}(A) = 2P_ABP_A$.

\begin{satz}[Hanson--Wright type inequalities for Stiefel manifolds]\label{HWI}
\begin{enumerate}
    \item For any $t \ge 0$, we have
    \[
    \mu_{n,d}(|f_2 - \mathrm{tr}(M)/n| \ge t) \le 2 \exp\Big(-\frac{1}{C}\min\Big(\frac{(n-2)^2t^2}{\lVert M \rVert_\mathrm{HS}^2}, \frac{(n-2)t}{\lVert M \rVert_\mathrm{op}}\Big)\Big),
    \]
    where we may choose $C = 128 e^2/\log(2)$.
    \item For any $t \ge 0$, we have
    \[
    \mu_{n,d}(|f_2 - \mathrm{tr}(M)/n| \ge t) \le 2 \exp\Big(-\frac{1}{C}\min\Big(\frac{(n-2)t^2}{\lVert P_AU \rVert_{\mathrm{HS},2}^2}, \frac{(n-2)t}{\lVert P_ABP_A \rVert_{\mathrm{op}, \infty}}\Big)\Big),
    \]
    where we may choose $C = 32 e^2/\log(2)$.
    \item For any $t \ge 0$, we have
    \[
    \mu_{n,d}(|f_2 - \mathrm{tr}(M)/n| \ge t) \le 2 \exp\Big(-\frac{1}{C}\min\Big(\frac{(n-2-8d)^2t^2}{\lVert P_ABP_A \rVert_{\mathrm{HS},2}^2}, \frac{(n-2)t}{\lVert P_ABP_A \rVert_{\mathrm{op}, \infty}}\Big)\Big),
    \]
    where we may choose $C = 256 e^2/\log(2)$.
\end{enumerate}
\end{satz}

For the notation used in (2) and (3), cf.\ \eqref{NotExpl}. The three different analogues of the classical Hanson--Wright inequality presented in Theorem \ref{HWI} reflect a certain flexibility which can turn out to be useful in applications (in particular, it would be hard to speak of a single canonical Stiefel variant of the Hanson--Wright inequality). Certainly, (1) is the most immediate analogue, involving quantities which can easily be calculated. However, unlike (2) and (3), it does not make use of intrinsic derivatives and can therefore be less accurate.

For instance, in the trivial case $M = I_{nd}$ (thus, $f_2(A) - \mathrm{tr}(M)/n \equiv 0$) (1) does not yield an optimal result, while (2) and (3) do. On the other hand, if $M_{ij,kl} = 1$ for every $i,j,k,l$, (1), (2) and (3) are not substantially different. As compared to (2), (3) is nearer to the classical Hanson--Wright inequality, but this is the cost of an additional $d$-dependence in the subgaussian term, which can be avoided in (2).

For explicit calculations of the $L^2$ norms appearing in (2) and (3), note that we may construct a uniformly distributed Stiefel element via $A=O_n I_{n,d} O_d$, where $O_n$ and $O_d$ are distributed according to the Haar measure on $SO(n)$ and on $SO(d)$, respectively, and $I_{n,d}$ is the element in $W_{n,d}$ whose columns are the first $d$ unit vectors. Then $P_A= A A^T = O_n I_{n,d} I_{n,d}^T O_n^{-1} = (\langle p_d \theta_j,  p_d \theta_k\rangle )_{j,k=1,\dots,n}$ (where $O_n=(\theta_1, \ldots, \theta_n)$ and $p_d$ is the projection to the first $d$ coordinates). Thus rewriting $\lVert P_A U\rVert_{\mathrm{HS},2}^2$  resp.\
$ \lVert P_A B P_A\rVert_{\mathrm{HS},2}^2 $ as linear combinations of products of components of orthogonal matrices, we may now  average them over $O(n)$ and $O(d)$ by applying formulas from the  Weingarten calculus. The only non-vanishing products in this average  are those whose indices are equal subject to a scheme of two pair partitions which will determine  a  so-called Weingarten function which is a rational function of $n$ of order $O(n^{-2})$ or smaller. For explicit formulas covering our cases
 see in particular \cite{CM09}, table  VII.  For an introduction to the underlying theory  we refer to \cite{CMN21} and for more details to e.\,g.\  \cite{CS06,BCS11}.

\begin{proof}[Proof of Theorem \ref{HWI}]
To see (1), we extend $f_2$ to $X \in \mathbb{R}^{n \times d}$ by the canonical choice $f_2(X) = \mathrm{vec}(X)^TM\mathrm{vec}(X)$ and apply Theorem \ref{ConckthOr} (2). Clearly, $f''(A) = 2M$, and in the subgaussian term, we use that by Poincar\'{e} inequality,
\[
\lVert \nabla f \rVert_{\mathrm{HS},2} \le \frac{4}{n-2} \lVert f'' \rVert_{\mathrm{HS},2}
\]
whenever $\nabla f$ has centered components (indeed, the proof is a simplification of the proof of Lemma \ref{1to2}). To see (2) and (3), it remains to apply Theorem \ref{Conc2ndOr} (2) and (3), respectively, using that $\nabla_W f(A) = 2P_AU$ and $f''_W(A) = 2P_ABP_A$.
\end{proof}

\subsection{Higher orders} Let us briefly comment on possible results for polynomial chaos of order $k \ge 3$. In general, similarly as in Theorem \ref{HWI} (1) one may derive concentration bounds for $f_k$ by applying Theorem \ref{ConckthOr} (for order $k$). Here, for $k=3$, by similar arguments as for orders $1$ and $2$ (including symmetry) we may use that $\int_{W_{n,d}} A_{ij}A_{kl} A_{pq} d\mu_{n,d} = 0$ for any choice of the indices. However, already for $k=4$ the calculation of the mixed moments (e.\,g.\ with each index appearing twice) can turn out to be quite involved.

Using the $L^p$ norm inequalities from Section \ref{sec:Sob}, it is also possible to derive results in the style of \cite{AW15} (not only for polynomial chaos but also for general $\mathcal{C}^k$ functions) which yield highly elaborate and accurate tail bounds in terms of a large family of tensor-type norms. We will not pursue this direction in this note.

Note though that all these results will make use of Euclidean derivatives only. In fact, the calculus of intrinsic derivatives gets increasingly involved as the order increases. For functions on the sphere, in \cite{BGS19} higher order concentration results depending on spherical partial derivatives have been derived, which can be thought of as a sort of intermediate object between Euclidean and intrinsic derivatives.

\subsection{Applications} Typical applications of concentration inequalities for Stiefel and Grassmann functionals include bounds for distances between subspaces. If $E,F \subset \mathbb{R}^n$ are two subspaces of dimension $d$, canonical choices for the distance between them are $\lVert P_E - P_F \rVert_\mathrm{HS}$ or $\lVert P_E - P_F \rVert_\mathrm{op}$. In the sequel, we will always choose one subspace at random, while the other one is fixed.

As a start, we consider the following simple example, in which the function $G_{n,d} \ni P \mapsto \lvert P - P_F \rvert \equiv \lVert P - P_F \rVert_\mathrm{HS}$ is studied.

\begin{proposition}\label{exampleGrSb}
Let $F \subset \mathbb{R}^n$ be any fixed $d$-dimensional subspace of $\mathbb{R}^n$, and denote by $P_F$ the corresponding projection matrix. Then, for any $t \ge 0$,
\[
\nu_{n,d}(\lvert \lvert P - P_F \rvert^2 - 2d(1-d/n) \rvert \ge t) \le 2 \exp(-(n-1)t^2/(64d)).
\]
\end{proposition}

\begin{proof}
Note that $\lvert P - P_F \rvert^2 = 2(d - \langle P, P_F \rangle)$, which has expectation $2d(1-d/n)$ since
\[
\int_{G_{n,d}} \langle P, P_F \rangle dP = d/n \mathrm{tr}(P_F) = d^2/n.
\]
Furthermore, the function $\mathbb{R}^{n \times n} \ni X \mapsto 2(d - \langle X, P_F \rangle)$ has derivative $-2 P_F$, which has Euclidean (Hilbert--Schmidt) norm $\sqrt{d}$. The result now follows immediately by combining \eqref{LipschitzAbl} (in the Grassmann version) and Proposition \ref{projGr}.
\end{proof}

It is possible to remove the square and consider fluctuations of $\lvert P - P_F \rvert$ by adapting some of the arguments used in the proof of Lemma \ref{transf} below. We skip the details, however, especially since we are rather interested in applying higher order concentration results in what follows.

To this end, let us consider another distance between two subspaces $E$ and $F$ of dimension $d$ which has been studied in \cite{WWF06} and subsequent papers. Here, we chose any orthonormal basis of $E$, say, $e_1, \ldots, e_d$, and set
\[
\mathrm{dist}(E,F) := \Big(\sum_{j=1}^d \mathrm{dist}(e_j, F)^2\Big)^{1/2},
\]
where for any $x \in \mathbb{R}^n$, $\mathrm{dist}(x,F) = \min_{y \in F} |x-y|$ is the usual point-to-subspace distance induced by the Euclidean norm. It is not hard to see that $\mathrm{dist}(E,F)$ does not depend on the choice of the orthonormal basis of $E$, cf.\ \cite[Theorem 1]{WWF06}.

As above, let us fix $F$ and consider the distance as a function of the subspace $E$. In view of the definition via orthonormal bases (and the invariance under change of the basis), we may then regard $\mathrm{dist}$ as a function on the Stiefel manifold $W_{n,d}$. That is, writing $A_{\bullet j}$ for the columns of $A \in W_{n,d}$, we shall study the fluctuations of
\[
\mathrm{dist}(A,F) := \Big(\sum_{j=1}^d \mathrm{dist}(A_{\bullet j}, F)^2\Big)^{1/2}
\]
under $\mu_{n,d}$. In this situation, the function fits into a framework which allows to derive concentration bounds via Hanson--Wright inequalities similarly as in \cite[Corollary 3.1]{RV13}, where distances of random vectors to a fixed subspace have been stuied. Under the uniform distribution on the Stiefel manifold, we obtain the following result.

\begin{proposition}\label{subspaces}
Let $F \subset \mathbb{R}^n$ be any fixed $d$-dimensional subspace of $\mathbb{R}^n$. Then, for any $t \ge 0$,
\[
\mu_{n,d} ( |\mathrm{dist}(A,F) - d/\sqrt{n}| \ge t) \le 2 \exp(-(n-2)t^2/C)
\]
with $C = 384 e^2/\log(2)$.
\end{proposition}

To prove Proposition \ref{subspaces}, we need the following lemma, which is itself of independent interest since it generalizes \cite[Theorem 2.1]{RV13} (cf.\ also \cite[Proposition 2.2]{S20+}). In essence, given any matrix $M \in \mathbb{R}^{nd \times nd}$ it controls the fluctuations of the Euclidean norm of $MA$ around the Hilbert--Schmidt norm of $M$, rescaled by $n^{-1/2}$ according to the standard deviation of the entries of $A \sim \mu_{n,d}$.

\begin{lemma}\label{transf}
For any fixed $M \in \mathbb{R}^{nd \times nd}$, we have
\[
\mu_{n,d}(|\lvert MA \rvert - \lVert M \rVert_\mathrm{HS}/\sqrt{n} | \ge t) \le 2 \exp\Big(-\frac{(n-2)t^2}{C\lVert M \rVert_\mathrm{op}^2}\Big)
\]
with $C = 384 e^2/\log(2)$.
\end{lemma}

\begin{proof}
Let $Q = M^TM$, so that $|MA|^2 = A^TQA$ (recall \eqref{shn} and \eqref{notquf}) and $\mathrm{tr}(Q) = \lVert M \rVert_\mathrm{HS}^2$. Assuming that $\lVert M \rVert_\mathrm{HS} = \sqrt{n}$, Theorem \ref{HWI} (1) yields
\begin{align*}
    \mu_{n,d}(|\lvert MA \rvert^2 - 1| \ge t) &\le 2 \exp\Big(-\frac{1}{C}\min\Big(\frac{(n-2)^2t^2}{n\lVert M \rVert_\mathrm{op}^2}, \frac{(n-2)t}{\lVert M \rVert_\mathrm{op}^2}\Big)\Big)\\
    &\le 2 \exp\Big(-\frac{n-2}{3C\lVert M \rVert_\mathrm{op}^2}\min(t^2,t)\Big),
\end{align*}
where the first step follows from $\norm{Q}_{\mathrm{HS}}^2 \le \norm{M}_{\mathrm{op}}^2 \norm{M}_{\mathrm{HS}}^2 = n \norm{M}_{\mathrm{op}}^2$ and $\norm{Q}_{\mathrm{op}} \le \norm{M}_{\mathrm{op}}^2$, and $C = 128 e^2/\log(2)$.

Now, as in \cite{RV13}, we use the inequality $\abs{z - 1} \le \min( \abs{z^2 -1}, \abs{z^2 - 1}^{1/2})$, giving for any $t \ge 0$
\[
\mu_{n,d} ( |\lvert MA \rvert - 1| \ge t \Big) \le \mu_{n,d} (| \lvert MA \rvert^2 - 1| \ge \max(t,t^2)).
\]
Combining this with the first step and using that $\min(\max(t,t^2)^2, \max(t,t^2)) = t^2$ yields
\[
\mu_{n,d} (|\lvert MA \rvert - 1| \ge t) \le 2 \exp\Big(-\frac{(n-2)t^2}{3C\lVert M \rVert_\mathrm{op}^2}\Big),
\]
i.\,e.\ the claim for for $\lVert M \rVert_{\mathrm{HS}} = \sqrt{n}$. The general case now follows by considering $\tilde{M} \coloneqq \sqrt{n} M \norm{M}_{\mathrm{HS}}^{-1}$, noting that
\[
\mu_{n,d} (|\lvert MA \rvert - \lVert M \rVert_{\mathrm{HS}}/\sqrt{n}| \ge t) = \mu_{n,d} ( |\lvert\tilde{M}A \rvert - 1| \ge \sqrt{n}t\lVert M \rVert_{\mathrm{HS}}^{-1}).
\]
\end{proof}

\begin{proof}[Proof of Proposition \ref{subspaces}]
Let $P_F$ be the orthogonal projection onto $F$, and denote by $A_{\bullet j}$ the columns of $A$. Then, $\mathrm{dist}(A_{\bullet j}, F) = |P_FA_{\bullet j}|$. In particular, if $M := I_d \otimes P_F$ is the $nd \times nd$ block diagonal matrix whose diagonal consists of $d$ copies of the matrix $P_F$, $\mathrm{dist}(A,F) = |MA|$, so that we may apply Lemma \ref{transf} (1). Here we use that $\lVert M \rVert_\mathrm{HS}^2 = d \lVert P_F \rVert_\mathrm{HS}^2 = d^2$ and moreover that since $d$ eigenvalues of $P_F$ are $1$ and all other eigenvalues are $0$, we have $\lVert M \rVert_\mathrm{op} = \lVert P_F \rVert_\mathrm{op} = 1$.
\end{proof}

In fact, it is not even necessary to consider subspaces $F$ of the same dimension $d$. If $F \subset \mathbb{R}^n$ is a subspace of any dimension $m < n$, the definitions of $\mathrm{dist}(E,F)$ and $\mathrm{dist}(A,F)$ continue to hold (even if $\mathrm{dist}(E,F)$ is no longer symmetric in $E$ and $F$ in this case), and with minimal adaptions in the proof of Proposition \ref{subspaces}, one can show that
\[
\mu_{n,d} ( |\mathrm{dist}(A,F) - \sqrt{md/n}| \ge t) \le 2 \exp(-(n-2)t^2/C)
\]
with $C = 384 e^2/\log(2)$.

\appendix
\section{Proofs of the results from Section \ref{sec:DerGr}}

\begin{proof}[Proof of Proposition \ref{Prop9.1.3Gr}]
Considering the function
\[
\psi_V(P) := \langle \nabla_G f(P), V \rangle = \langle [P,[P,\nabla f(P)]], V \rangle,
\]
let us calculate $\nabla \psi_V(P)$. Choosing some extension of $f$ onto an open neighborhood of $G_{n,d}$ in $\mathbb{R}^{n \times n}$, the function
\begin{align*}
\psi_V(X) &:= \langle [X,[X,\pis(\nabla f(X))]], V \rangle\\
&= \mathrm{tr}(V^T(X\pis(\nabla f(X)) + \pis(\nabla f(X))X - 2 P\pis(\nabla f(X)) X))
\end{align*}
is a smooth extension of $\psi_V(P)$. Note that $\pis$ has derivative
\[
\frac{d\pis(Y)}{dY} = \frac{1}{2}\frac{d(Y + Y^T)}{dY} = \frac{1}{2}(I_{n^2} + K_{n,n}).
\]
As a consequence, by chain rule,
\[
\frac{d\pis(\nabla f(X))}{dX} = \frac{1}{2}(I_{n^2} + K_{n,n})f''(X).
\]

Let us now calculate
\[
\frac{d}{dX}[V^T(X\pis(\nabla f(X)) + \pis(\nabla f(X))X - 2 X\pis(\nabla f(X)) X)].
\]
First, by Leibniz rule,
\begin{align*}
&\quad \frac{d}{dX} [V^TX\pis(\nabla f(X))]\\
&= (\pis(\nabla f(X)) \otimes I_n) (I_n \otimes V^T) + \frac{1}{2}(I_n \otimes (V^TX)) (I_{n^2} + K_{n,n})f''(X)\\
&= \pis(\nabla f(X)) \otimes V^T + \frac{1}{2}(I_n \otimes (V^TX)) (I_{n^2} + K_{n,n})f''(X).
\end{align*}
Next, note that similarly,
\[
\frac{d(\pis(\nabla f(X))X)}{dX} = \frac{1}{2} (X^T \otimes I_n) (I_{n^2} + K_{n,n})f''(X) + I_n \otimes \pis(\nabla f(X)).
\]
In particular, we obtain that
\begin{align*}
&\quad \frac{d}{dX}[V^T\pis(\nabla f(X))X]\\
&= (I_n \otimes V^T)\Big(\frac{1}{2} (X^T \otimes I_n) (I_{n^2} + K_{n,n})f''(X) + I_n \otimes \pis(\nabla f(X))\Big)\\
&= \frac{1}{2} (X^T \otimes V^T) (I_{n^2} + K_{n,n})f''(X) + I_n \otimes (V^T\pis(\nabla f(X))).
\end{align*}
Finally,
\begin{align*}
&\quad \frac{d}{dX}[V^TX\pis(\nabla f(X)) X]\\
&= ((X^T\pis[\nabla f(X)]) \otimes I_n) \frac{d(V^TX)}{dX} + (I_n \otimes (V^TX)) \frac{d}{dX}[\pis(\nabla f(X)) X]\\
&= ((X^T\pis[\nabla f(X)]) \otimes I_n) (I_n \otimes V^T)\\
&\quad + (I_n \otimes (V^TX)) \Big(\frac{1}{2} (X^T \otimes I_n) (I_{n^2} + K_{n,n})f''(X) + I_n \otimes \pis(\nabla f(X))\Big)\\
&= (X^T\pis[\nabla f(X)]) \otimes V^T + \frac{1}{2} (X^T \otimes (V^TX)) (I_{n^2} + K_{n,n})f''(X)\\
&\quad + I_n \otimes (V^TX\pis(\nabla f(X))).
\end{align*}

Consequently, by chain rule,
\begin{align*}
    &\quad \frac{d}{dX}\mathrm{tr}(V^T(X\pis(\nabla f(X)) + \pis(\nabla f(X))X - 2 X\pis(\nabla f(X)) X))\\
    &= \vec(I_n)^T\Big(\pis(\nabla f(X)) \otimes V^T + \frac{1}{2}(I_n \otimes (V^TX)) (I_{n^2} + K_{n,n})f''(X)\\
    &\quad + \frac{1}{2} (X^T \otimes V^T) (I_{n^2} + K_{n,n})f''(X) + I_n \otimes (V^T\pis(\nabla f(X)))\\
    &\quad - 2(X^T\pis[\nabla f(X)]) \otimes V^T - (X^T \otimes (V^TX)) (I_{n^2} + K_{n,n})f''(X)\\
    &\quad - 2I_n \otimes (V^TX\pis(\nabla f(X)))\Big).
\end{align*}
Switching from $d/(dX)$ to $\vec(\nabla)$ (and thus, transposing), we obtain
\begin{align*}
    \vec(\nabla\psi_V(X))
    &= \Big(\pis(\nabla f(X)) \otimes V + \frac{1}{2}f''(X)(I_{n^2} + K_{n,n})(I_n \otimes (X^TV))\\
    &\quad + \frac{1}{2} f''(X)(I_{n^2} + K_{n,n})(X \otimes V) + I_n \otimes (\pis(\nabla f(X))V)\\
    &\quad - 2(\pis(\nabla f(X))X) \otimes V - f''(X) (I_{n^2} + K_{n,n})(X \otimes (X^TV))\\
    &\quad - 2I_n \otimes (\pis(\nabla f(X))X^TV)\Big)\vec(I_n).
\end{align*}

Let us simplify this further. We have
\begin{align*}
(\pis(\nabla f(X)) \otimes V)\vec(I_n) &= \vec[V\pis(\nabla f(X))],\\
(I_n \otimes (\pis(\nabla f(X))V))\vec(I_n) &= \vec[(\pis(\nabla f(X))V)],\\
((\pis(\nabla f(X))X) \otimes V)\vec(I_n) &= \vec[VX^T\pis(\nabla f(X))],\\
(I_n \otimes (\pis(\nabla f(X))X^TV))\vec(I_n) &= \vec[\pis(\nabla f(X))X^TV)].
\end{align*}
Recalling that $\pis = (I_{n^2} + K_{n,n})/2$, we may further rewrite
\begin{align*}
\frac{1}{2}f''(X)(I_{n^2} + K_{n,n})(I_n \otimes (X^TV))\vec(I_n) &= f''(X)\pis \vec(X^TV),\\
\frac{1}{2} f''(X)(I_{n^2} + K_{n,n})(X \otimes V)\vec(I_n) &= f''(X)\pis \vec(VX^T),\\
f''(X) (I_{n^2} + K_{n,n})(X \otimes (X^TV))\vec(I_n) &= f''(X)\pis \vec(X^TVX^T)
\end{align*}
Summing up and restricting to $P \in G_{n,d}$ (which is symmetric, in particular),
\begin{align*}
    \vec(\nabla\psi_V(P))
    &= f''(P)\pis\vec([P,[P,V]]) + 2 \vec[\pis(V\pis(\nabla f(P)))]\\
    &\quad - 4\vec[\pis(VP\pis(\nabla f(P)))].
\end{align*}

Next we calculate the intrinsic derivatives. First, we need to project onto $\Rns$ by applying $\pis$. Using the notation from Proposition \ref{2ndOrDerGr}, this yields
\begin{align*}
    \pis\nabla\psi_V(P)
    &= \pis f''(P)\pis\vec([P,[P,V]]) + 2 \vec[\pis(V\pis(\nabla f(P)))]\\
    &\quad - 4\vec[\pis(VP\pis(\nabla f(P)))].
\end{align*}
Now we project onto $T_P$ by applying $\pi_P$. This yields
\begin{align*}
    \nabla_G\psi_V(P)
    &= \pi_P\pis f''(P)\pis\pi_P(V) + 2\pi_P \vec[\pis(V\pis(\nabla f(P)))]\\
    &\quad - 4\pi_P\vec[\pis(VP\pis(\nabla f(P)))].
\end{align*}
By Proposition \ref{2ndOrDerGr}, it follows that
\begin{align*}
    f_G''(P)V &= \nabla_G\psi_V(P) - 2\pi_P \vec[\pis(V\pis(\nabla f(P)))]\\
    &\quad + 4\pi_P\vec[\pis(VP\pis(\nabla f(P)))] - [P,[\pis(\nabla f(P)),\pi_PV]].
\end{align*}
Now, a direct calculation yields
\begin{align*}
    &\quad - 2\pi_P \vec[\pis(V\pis(\nabla f(P)))] + 4\pi_P\vec[\pis(VP\pis(\nabla f(P)))]\\
    &\quad - [P,[\pis(\nabla f(P)),\pi_PV]]\\
    &= - [P,[\pi_P \nabla f(P), V]] = - [P,[\nabla_G f(P), V]],
\end{align*}
which finishes the proof.
\end{proof}

\begin{proof}[Proof of Proposition \ref{Prop9.2.1Gr}]
Similarly as in the proof of Proposition \ref{Prop9.2.1St}, it follows from Proposition \ref{Prop9.1.3Gr} applied to any $V \in \Rns$ such that $|V| \equiv \lVert V \rVert_\mathrm{HS} = 1$ together with compactness arguments that $|\nabla_G f|$ has finite Lipschitz semi-norm.

To show the identity for second order modulus of gradient, fix $P \in G_{n,d}$ and note that by the definition of the intrinsic gradient and Proposition \ref{Prop9.1.3Gr}, we have
\[
\langle \nabla_Gf(P'),V\rangle = \langle \nabla_Gf(P),V\rangle + \langle \tilde{V}, P'-P\rangle + o(|P'-P|),
\]
where
\[
\tilde{V} = f''_G(P)V + [P,[\nabla_Gf(P),V]].
\]
Moreover, using the integral form of the Taylor formula and the compactness of $G_{n,d}$, which implies that every continuous function is already uniformly continuous, we see that the remainder term in the Taylor expansion can be bounded independently of $V \in \Rns$ such that $|V| = 1$, i.\,e.
\[
\sup_{|V| = 1} |\langle \nabla_Gf(P'),V\rangle - \langle \nabla_Gf(P),V\rangle - \langle \tilde{V}, P'-P\rangle| \le \varepsilon(|P'-P|),
\]
where $\varepsilon(t)$ is some function which satisfies $\varepsilon(t) \to 0$ as $t \to 0$.

We now rewrite the Taylor formula as
\begin{equation}\label{TayFGr}
\langle \nabla_Gf(P'),V\rangle = \langle \nabla_Gf(P) + L,V\rangle + o(|P'-P|)
\end{equation}
by choosing a suitable $L$ such that $\langle L, V \rangle = \langle \tilde{V}, P'-P\rangle$. Obviously,
\[
\langle f_G''(P)V, P'-P \rangle = \langle f''_G(P)(P'-P), V \rangle.
\]
Moreover, write
\[
[P,[\nabla_Gf(P),V]] = P\nabla_Gf(P)V - PV\nabla_Gf(P) - \nabla_Gf(P)VP + V\nabla_Gf(P)P.
\]
Adapting the identities \eqref{genrul} to $U,W \in \mathbb{R}^{n \times n}$ and $V,X (= P'-P) \in \Rns$, we obtain
\[
\langle U^TXW^T, V \rangle = \langle UVW, X \rangle = \langle WXU, V \rangle,
\]
where we have used the symmetry of $V$ and $X$. If we apply the first identity to $P\nabla_Gf(P)V$ and $PV\nabla_Gf(P)$ and the second one to $\nabla_Gf(P)VP$ and $V\nabla_Gf(P)P$, we obtain
\[
\langle [P,[\nabla_Gf(P),V]], P'-P \rangle = 2 \langle [\nabla_Gf(P),P(P'-P)], V \rangle,
\]
so that altogether, we may also set
\[
L := f''_G(P)(P'-P) + 2 [\nabla_Gf(P),P(P'-P)].
\]

Now we take an absolute value on both sides of the Taylor formula \eqref{TayFGr} and take the supremum over all symmetric $V$ such that $\mathrm{vec}(V) \in S^{nd-1}$. This leads to
\[
|\nabla_Gf(P')| = |\nabla_Gf(P) + L| + o(|P'-P|).
\]
Next, we write
\[
|\nabla_Gf(P) + L|^2 = |\nabla_Gf(P)|^2 + 2 \langle \nabla_Gf(P), L \rangle + |L|^2.
\]
Note that as a trace of the product of a symmetric and an anti-symmetric matrix, $\langle \nabla_Gf(P),[\nabla_Gf(P),P(P'-P)] \rangle = 0$. Therefore,
\[
\langle \nabla_G f(P), L \rangle = \langle \nabla_Gf(P), f_G''(P)(P'-P) \rangle = \langle U, P'-P \rangle
\]
with $U := f_G''(P)\nabla_Gf(P)$. Since $|L|^2 = O(|P'-P|^2)$, we obtain
\[
|\nabla_Gf(P) + L|^2 = |\nabla_Gf(P)|^2 + 2 \langle U, P'-P \rangle + o(|P'-P|).
\]

If $|\nabla_Gf(P)| > 0$, it therefore follows that
\[
|\nabla_Gf(P) + L| = |\nabla_Gf(P)| + |\nabla_Gf(P)|^{-1} \langle U, P'-P \rangle + o(|P'-P|).
\]
Hence,
\[
|\nabla_Gf(P')| - |\nabla_Gf(P)| = |\nabla_Gf(P)|^{-1} \langle U, P'-P \rangle + o(|P'-P|)
\]
and thus
\begin{align*}
\limsup_{P'\to P} \frac{\big||\nabla_Gf(P')| - |\nabla_Gf(P)|\big|}{|P'-P|}
&= |\nabla_Gf(P)|^{-1} \limsup_{P' \to P} \frac{\big|\langle U, P'-P \rangle\big|}{|P'-P|}\\
&= |\nabla_Gf(X)|^{-1} |\nabla_G\psi_U(P)|,
\end{align*}
where $\psi_U (P) := \langle U, P \rangle$. As noted after Proposition \ref{2ndOrDerGr}, it holds that $U \in T_P$, so that $\nabla_G \psi_U(P) = U$. Thus, we arrive at
\[
|\nabla^{(2)}_Gf(P)| = |\nabla_Gf(P)|^{-1} |f_G''(P)\nabla_Gf(P)|
\]
if $|\nabla_Wf(P)| > 0$.

It remains to consider the case where $|\nabla_Gf(P)| = 0$. Here, $L = f_G''(P)(P'-P)$, and the Taylor formula reads
\[
|\nabla_G f(P')| = |L| + o(|P'-P|).
\]
It follows that
\begin{align*}
    |\nabla^{(2)}_Gf(P)| &= \limsup_{P' \to P} \frac{|\nabla_Gf(P')|}{|P'-P|}
    = \limsup_{P' \to P} \frac{|f_G''(P)(P'-P)|}{|P'-P|}\\
    &= \limsup_{V \to 0, V \in T_P^\perp} \frac{|f_G''(P)V|}{|V|} = \lVert f_G''(P) \rVert_\mathrm{op},
\end{align*}
which finishes the proof.
\end{proof}

\end{document}